\documentclass[twoside,a4paper]{amsart}

\usepackage{xcolor}
\definecolor{webgreen}{rgb}{0,.5,0}
\definecolor{webbrown}{rgb}{.6,0,0}
\definecolor{RoyalBlue}{cmyk}{1, 0.50, 0, 0}
\usepackage[colorlinks=true, breaklinks=true, urlcolor=webbrown, linkcolor=RoyalBlue, citecolor=webgreen,backref=page]{hyperref}

\usepackage{epsfig, graphicx, subfigure}
\usepackage{verbatim, setspace}
\usepackage{amsmath, amssymb}

\usepackage{mathptmx}
\usepackage{newtxtext,newtxmath}

\newcommand{\T}		{\mathbb{T}}
\newcommand{\D}		{\mathbb{D}}

\newcommand{\C}		{\mathbb{C}}
\newcommand{\N}		{\mathbb{N}}
\newcommand{\Z}		{\mathbb{Z}}

\newcommand{\qandq}{\quad \text{and} \quad}
\newcommand{\qasq}{\quad \text{as} \quad}

\newcommand{\dd}{\mathrm d}
\newcommand{\ic}{\mathrm i}
\newcommand{\cp}{\mathrm{cp}}

\newcommand{\Qn}{{\mathcal Q}_n}
\newcommand{\Rn}{{\mathcal R}_n}

\newcommand{\RS}{\boldsymbol{\mathfrak{S}}}
\newcommand{\am}{\mathfrak a}

\newcommand{\z}{\boldsymbol z}
\newcommand{\s}{\boldsymbol s}
\newcommand{\p}{\boldsymbol p}

\newcommand{\im}{\mathrm{Im}}
\newcommand{\re}{\mathrm{Re}}

\newcommand{\rhy}   {\textnormal{RHP}-${\boldsymbol Y}$}
\newcommand{\rhx}   {\textnormal{RHP}-${\boldsymbol X}$}
\newcommand{\rhz}   {\textnormal{RHP}-${\boldsymbol Z}$}
\newcommand{\rhn}   {\textnormal{RHP}-${\boldsymbol N}$}
\newcommand{\rhp}   {\textnormal{RHP}-$\protect\boldsymbol P_e$}
\newcommand{\rhpsiA}   {\textnormal{RHP}-${\boldsymbol \Psi_\alpha}$}
\newcommand{\rhpsi}   {\textnormal{RHP}-${\boldsymbol \Psi}$}

\newtheorem{theorem}{Theorem}
\newtheorem{proposition}[theorem]{Proposition}

\newtheorem{lemma}[theorem]{Lemma}

\newtheorem*{st}{Theorem (Stahl)}
\newtheorem*{bt}{Theorem (Buslaev)}

\theoremstyle{remark}

\begin{document}

\title[Two-point Pad\'e approximants to piecewise holomorphic functions]{Convergence of two-point Pad\'e approximants to piecewise holomorphic functions}

\author[M. Yattselev]{Maxim L. Yattselev}

\address{Department of Mathematical Sciences, Indiana University-Purdue University Indianapolis, 402~North Blackford Street, Indianapolis, IN 46202}

\address{Keldysh Institute of Applied Mathematics, Russian Academy of Science, Miusskaya Pl. 4, Moscow, 125047 Russian Federation}

\email{\href{mailto:maxyatts@iupui.edu}{maxyatts@iupui.edu}}

\thanks{The research was supported in part by a grant from the Simons Foundation, CGM-354538, and by Moscow Center for Fundamental and Applied Mathematics, Agreement with the Ministry of Science and Higher Education of the Russian Federation, No. 075-15-2019-1623.}



\maketitle

\begin{abstract}
Let \( f_0 \) and \( f_\infty \) be formal power series at the origin and infinity, and \( P_n/Q_n \), with \( \deg(P_n),\deg(Q_n)\leq n \), be a rational function that simultaneously interpolates \( f_0 \) at the origin with order \( n \) and \( f_\infty \) at infinity with order \( n+1 \). When germs \( f_0,f_\infty \) represent multi-valued functions with finitely many branch points, it was shown by Buslaev \cite{Bus13} that there exists a unique compact set \( F \) in the complement of which the approximants converge in capacity to the approximated functions. The set \( F \) might or might not separate the plane. We study  uniform convergence of the approximants for the geometrically simplest sets \( F \) that do separate the plane.
\end{abstract}

\maketitle

\section{Introduction}
\label{s1}

Let \( f(z) = \sum_{k=0}^\infty f_kz^{-k} \) be a formal power series at infinity and \( M_n/N_n \) be a rational function such that $\deg(M_n),\deg(N_n)\leq n$ and
\[
(N_nf-M_n)(z) = \mathcal{O}\big(z^{-n-1}\big) \qasq z\to\infty.
\]
It is known that such a rational function is unique and is called the \emph{classical diagonal Pad\'e approximant} to \( f \) at infinity. The following theorem\footnote{In the introduction, only abridged statements of the theorems that suits our needs are stated.} summarizes one the most fundamental contributions of Herbert Stahl to complex approximation theory \cite{St85,St85b,St86,St97}.

\begin{st}
Assume that the germ at infinity \( f \) can be analytically continued along any path in \( \overline\C \) that avoids a fixed polar set\footnote{See \cite{Ransford} for notions of potential theory.} and there is at least one point outside of this set with at least two distinct continuations. Then there exists a compact set \( F \) such that
\begin{itemize}
\item[(i)] \( F \) does not separate the plane and \( f \) has a holomorphic and single-valued extension into the domain \( D:=\overline\C\setminus F \);
\item[(ii)] \( F \) consists of open non-intersecting analytic arcs \( J_i \), their endpoints, and a subset of the singular set of \( f \), and\footnote{Such sets \( F \) are now called \emph{symmetric contours} or \emph{S-curves}.}
\[
\frac{\partial g_F(z,\infty)}{\partial n_+} = \frac{\partial g_F(z,\infty)}{\partial n_-}
\]
at each point \( z\in\cup_i J_i \), where \( g_F(\cdot,\infty) \) is the Green function for \( D \) with pole at infinity and \( \partial/\partial n_\pm \) are the one-sided normal derivatives;
\item[(iii)] it holds for any compact set \( V\subset D \) that 
\[
\lim_{n\to\infty}\cp\left\{z\in V :\big|(f-M_n/N_n)(z)\big|^{1/2n}\geq e^{-g_F(z,\infty)}+\epsilon\right\} = 0
\]
for any \( \epsilon>0 \), where \( \cp(\cdot) \) is the logarithmic capacity.
\end{itemize}
\end{st}

More generally, if we select a branch of \( f \) which is holomorphic in a neighborhood of a certain closed set as well as a collection of \( 2n+1 \) not necessarily distinct points in this set, one can define a multipoint Pad\'e approximant interpolating \( f \) at these points. An analog of Stahl's theorem for multipoint Pad\'e approximants was proven by Gonchar and Rakhmanov \cite{Grakh87}. However, the existence of the set \( F \) satisfying (i) and an appropriate generalization of (ii) was not shown but assumed in \cite{Grakh87}, from which the conclusion (iii) was then obtained (see \cite{Rakh12,BStY12,uY} for results on existence of such weighted symmetric contours). 

Weighted symmetric contours are characterized as contours minimizing certain weighted logarithmic capacity.  One of the major obstructions in proving a general theorem on their existence lies in the fact that a minimizer can separate the plane. In \cite{Bus13}, Buslaev treated this possibility not as a hindrance but as an important feature. More precisely, let \( f_0 \) and \( f_\infty \) be formal power series at the origin and infinity, respectively. That is,
\begin{equation}
\label{fs}
f_0(z) := \sum_{k=0}^\infty f_{k,0}z^k \quad \text{and} \quad f_\infty(z) := \sum_{k=0}^\infty f_{k,\infty}z^{-k}.
\end{equation}
A rational function $P_n/Q_n$ is a \emph{two-point Pad\'e approximant of type $(n_1,n_2)$}, $n_1+n_2=2n+1$, to the pair $(f_0,f_\infty)$ if $\deg(P_n),\deg(Q_n)\leq n$ and
\begin{equation}
\label{Pade}
\left\{
\begin{array}{ll}
(Q_nf_0-P_n)(z) = \mathcal{O}(z^{n_1}), & z\to 0, \medskip \\
(Q_nf_\infty-P_n)(z) = \mathcal{O}(z^{n-n_2}), & z\to\infty.
\end{array}
\right.
\end{equation}
As in the case of the classical Pad\'e approximants, the ratio \( P_n/Q_n \) is always unique. In \cite[Theorem~1]{Bus13}, see also \cite{Bus15} for a generalization to \( m \)-point Pad\'e approximants, Buslaev proved the following.

\begin{bt}
Assume that the germs \( f_0 \) and \( f_\infty \) in \eqref{fs} can be analytically continued along any path in \( \overline\C \) that avoids finitely many fixed points one of which is a branch point of \( f_0 \) and another is a branch point of \( f_\infty \). Then there exists a compact set \( F \) such that
\begin{itemize}
\item[(i)] \( \overline\C\setminus F =D_0\cup D_\infty \), where the domains \( D_0\ni 0 \) and \( D_\infty \ni\infty \) either do not intersect or coincide, and \( f_e \) has a holomorphic and single-valued extension into \( D_e \), \( e\in\{0,\infty\} \) (if \( D_0=D_\infty=D \), then \( f_0,f_\infty \) are analytic continuations of each other within \( D \));
\item[(ii)] \( F \) consists of open analytic arcs and their endpoints and at each point of these arcs it holds that
\[
\frac{\partial \big(g_F(z,0)+g_F(z,\infty)\big)}{\partial n_+} = \frac{\partial \big(g_F(z,0)+g_F(z,\infty)\big)}{\partial n_-},
\]
where \( g_F(z,0) \) is the Green function for \( D_0 \) with pole at \( 0 \) and \( g_F(z,\infty) \) is the Green function for \( D_\infty \) with pole at infinity;
\item[(iii)] if indices \( n_1+n_2=2n+1 \) in \eqref{Pade} are such that \( \lim_{n\to\infty}n_i/n=1 \), \( i\in\{1,2\} \), and \( \partial D_0\not\subset\partial D_\infty \) together with \( \partial D_\infty\not\subset\partial D_0 \), then it holds on any compact set \( V \subset \C\setminus F \) that
\[
\lim_{n\to\infty}\cp\left\{z\in V :\big|(f-P_n/Q_n)(z)\big|^{1/n}\geq e^{-g_F(z,0)-g_F(z,\infty)}+\epsilon\right\} = 0
\]
for any \( \epsilon>0 \), where \( \cp(\cdot) \) is the logarithmic capacity and \( f=f_e \) in \( D_e \), \( e\in\{0,\infty\} \).
\end{itemize}
\end{bt}

Let \( f \) be a germ at infinity that can be continued analytically along any path in \( \overline\C \) avoiding finitely many points one of which is  a branch point for some continuation. Set \( f_\infty = f \) and \( f_0 \) to be one of the function elements of \( f \) at the origin. Buslaev's result tells us that independently of the choice of \( f_0 \), the two-point Pad\'e approximants always converge to the approximated pair of germs. However, the region of the convergence can  consist of either one or two components.

The conclusion (iii) of Buslaev's theorem describes the so-called \emph{weak} or \( n \)-th root asymptotics of Pad\'e approximants. The goal of the present work is to establish \emph{strong} asymptotic formulae along appropriate subsequences of indices. In the case where \( F \) does not separate the plane such results were obtained in \cite{LL_OPA88,St99,BY09c,BY10,uY} under various assumptions on the approximated function. Below, we study the case of distinct germs in the simplest geometrical setting, see~Figure~\ref{F}.

\section{Main Results}

Given $a\in\D\setminus\{0\}$, denote by $K$ the \emph{Chebotar\"ev compact} for $\{-1,1,\mathfrak j(a)\}$, where \( \D \) is the open unit disk centered at the origin and $\mathfrak j(z):=(z+z^{-1})/2$ is the Jukovski map. $K$ consists of the critical trajectories of the quadratic differential
\[
-\frac{(\zeta-\mathfrak j(b))\dd \zeta^2}{(\zeta^2-1)(\zeta-\mathfrak j(a))},
\]
where $b\in\D$ is the uniquely determined point (\( \mathfrak j(b) \) is called Chebotar\"ev's center). When \( \im(a)\neq 0 \), the set $K$ consists of three analytic arcs emanating from $\mathfrak j(b)$ and terminating at each of the points $-1,1,\mathfrak j(a)$. When \( a\in(-1,0) \), it holds that \( b=-1 \) and \( K=[-\mathfrak j(a),1] \) while \( b=1 \) and \( K=[-1,\mathfrak j(a)] \) for \( a\in (0,1) \). Buslaev's compact \( F \) corresponding to $a$ in this case is given by
\[
F := \big\{z:\mathfrak j(z)\in K\big\},
\]
see Figures~\ref{F}, \ref{N-real}, and \ref{N-complex}. 
\begin{figure}[ht!]
\centering
\subfigure[]{\includegraphics[scale=.4]{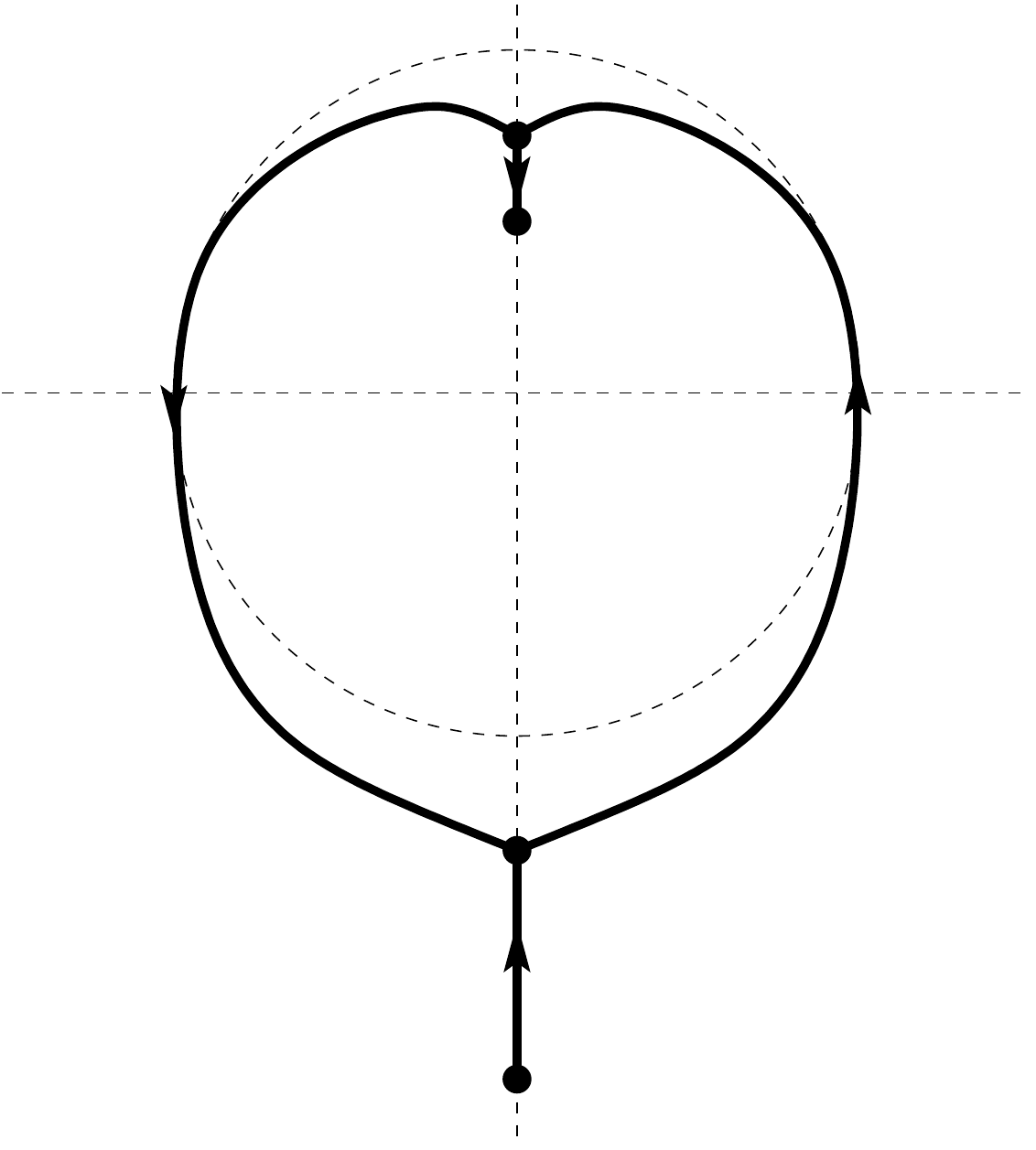}}
\begin{picture}(0,0)
\put(-23,85){$F_1$}
\put(-130,85){$F_{-1}$}
\put(-67,120){$F_a$}
\put(-67,20){$F_{a^{-1}}$}
\put(-71,110){$a$}
\put(-71,132){$b$}
\put(-71,42){$b^{-1}$}
\put(-71,0){$a^{-1}$}
\end{picture}
\quad\quad
\subfigure[]{\includegraphics[scale=.4]{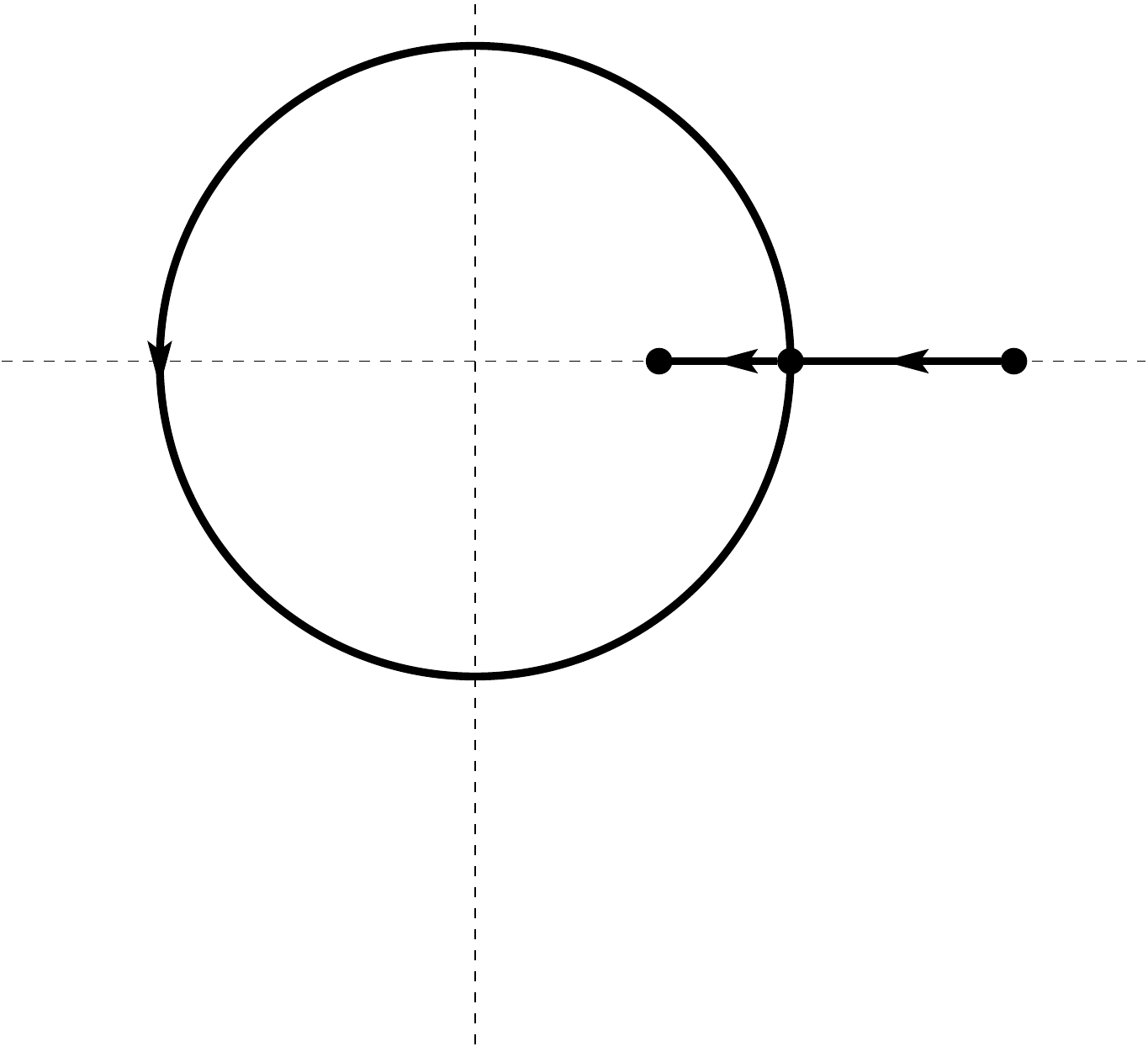}}
\begin{picture}(0,0)
\put(-157,85){$F_{-1}$}
\put(-47,85){$F_{a^{-1}}$}
\put(-68,85){$F_a$}
\put(-73,100){$a$}
\put(-27,100){$a^{-1}$}
\end{picture}
\caption{\small Buslaev's compact $F$ for  a point \( a \) on (a) the imaginary axis (b) the real axis.}
\label{F}
\end{figure}
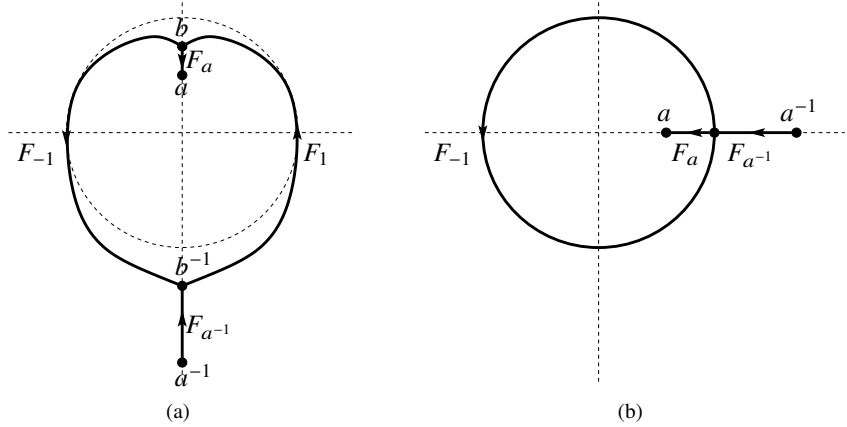
Clearly, $F$ separates the plane into two simply connected components, one containing the origin and one containing the point at infinity, which we label by $D_0$ and $D_\infty$.  We shall write
\[
F := F^\circ\cup E, \quad E:=\big\{a,a^{-1},b,b^{-1}\big\},
\]
where $F^\circ$ consists of four (three when \( a \) is real) open disjoint Jordan arcs, see Figure~\ref{F}. We let \( F_{\pm1} \) (resp. \( F_{\pm1}^\circ \)) to be closed (resp. open) Jordan arcs connecting \( b \) to \( b^{-1} \) and containing \( \pm1 \), and \( F_{a^{\pm1}} \) (resp. \(F_{a^{\pm1}}^\circ\)) to be closed (resp. open) Jordan arcs connecting \( a^{\pm1} \) to \( b^{\pm1} \) (we have that \( F_{-1}=\varnothing \) when \( a\in (-1,0) \) and \( F_1=\varnothing \) when \( a\in(0,1) \)). We choose their orientation so that \( F_1\cup F_{-1} \) is a counterclockwise oriented Jordan curve and \( F_{a^{-1}}\cup F_1\cup F_a \) is a Jordan arc oriented from \(a^{-1} \) to \( a \), see  Figure~\ref{F}. It is a matter of a simple substitution to see that the set $F^\circ$ is comprised of the critical trajectories of the quadratic differential
\begin{equation}
\label{diff}
-\frac{(z-b)(z-b^{-1})}{(z-a)(z-a^{-1})}\frac{\dd z^2}{z^2}.
\end{equation}

From now on we fix \( a\in \D\setminus\{0\} \) and the corresponding Buslaev's compact \( F \). In this work we shall be interested in approximating pairs \eqref{fs} of the form
\begin{equation}
\label{f}
\big( f_{\rho|D_0}, f_{\rho|D_\infty}\big), \quad f_\rho(z) := \frac{1}{2\pi\ic}\int_F\frac{\rho(s)\dd s}{s-z},
\end{equation}
for some classes of weights \( \rho \) to be specified later. Thus, depending on a situation we shall speak about approximants \eqref{Pade} to either a pair of functions or a single function, in both cases referring to \eqref{f}. As one can clearly see from Figure~\ref{F}, the structure of the set \( F \) is qualitatively different for real and non-real values of \( a \). Hence, it should come as no surprise that the description of the asymptotics of the approximants will also depend on whether the parameter \( a \) is real or not.

\subsection{Asymptotics of the Approximants: \( a\in(-1,0)\cup(0,1) \)}

Recall that in this case \( F \) is a union of the unit circle \( \T \) and the interval joining \( a \) and \( a^{-1} \) that we shall denote by \( [a^{-1},a] \) (always oriented from \( a^{-1} \) to \( a \)). Let us start by describing the class of weights \( \rho \) that we consider. To this end, define
\begin{equation}
\label{wa}
w(z) = w(z;a) := \sqrt{(z-a)(z-a^{-1})}
\end{equation}
to be the branch holomorphic off \( [a^{-1},a] \) such that \( w(z) = z + \mathcal O(1) \) as \( z\to\infty \). Given a function \( \rho \) on \( F \) it will be convenient for us to set\footnote{Given a function \( g \) holomorphic off an oriented Jordan arc or curve \( J \), we denote by \( g_+ \) (resp. \( g_- \)) the traces of \( g \) on the positive (resp. negative) side of \( J \) with respect to the orientation of \( J \).}
\begin{equation}
\label{h-real}
h(s) := \rho(s)\left\{
\begin{array}{rl}
w_\pm(s), & s\in F_{a^{\pm1}}^\circ, \medskip \\
w(s), & s\in F_{-b}^\circ,
\end{array}
\right.
\end{equation}
where \( b:=a/|a| \). We shall say that \( \rho\in\mathcal W_1 \) if \( h \) extends to a holomorphic and non-vanishing function in some neighborhood of \( F\setminus\big\{a,a^{-1}\big\} \) with the zero increment of the argument along the unit circle and there exist real constants \( \alpha,\beta>-1\) such that
\begin{equation}
\label{W1}
h_{|(a^{-1},a)}(s) = \widetilde h(s)(s-a)^{\alpha+1/2}(s-a^{-1})^{\beta+1/2},
\end{equation}
where \( \widetilde h \) is holomorphic and non-vanishing in some neighborhood of \( [a^{-1},a] \) and the branches of the power functions are holomorphic across \( (a^{-1},a) \). It will also be convenient to single out the following subclass of \( \mathcal W_1 \): we shall say that \( \rho\in \mathcal W_2 \subset \mathcal W_1 \) if \( \alpha=\beta=-1/2 \), that is, if \( h \) extends to a holomorphic and non-vanishing function in some neighborhood of \( F \). In particular, if we define \( \rho \) by \eqref{h-real} with \( h(s)\equiv 2 \), we get that the approximated pair of functions is given by \( (1/w,-1/w) \). 

For brevity, let us set \( \sqrt x := \ic \sqrt{|x|} \) when \( x<0 \). As we shall see below, the \( n \)-th root behavior of the approximants is described by the following function:
\begin{equation}
\label{varphi}
\varphi(z) := \frac{z+b+w(z)}{\sqrt{a}+\sqrt{a^{-1}}},
\end{equation}
which is holomorphic off \( [a^{-1},a] \), vanishes at the origin, and has a simple pole at infinity. Moreover, since \( w \) has purely imaginary traces on the interval \( [a^{-1},a] \) and \( w^2(e^{\ic t}) = 2e^{\ic t}(\cos(t)-\mathfrak j(a)) \), it can be easily checked that
\begin{equation}
\label{equilibrium}
\left\{
\begin{array}{rl}
\varphi_+(s)\varphi_-(s)=s, & s\in [a^{-1},a], \medskip \\
|\varphi(s)|=1, & s\in F_{-b}^\circ.
\end{array}
\right.
\end{equation}
To describe the finer behavior of the approximants let us define the following functions. Let \( \log h \) be any smooth continuous branch of the logarithm on the unit circle (recall that \( h \) is continuous on  \( \T \) with the argument that has zero increment). Set
\begin{equation}
\label{Sh-1}
D(z) := \exp\left\{\frac{1}{2\pi\ic}\int_{\T}\frac{\log h(s)\dd s}{s-z}\right\},
\end{equation}
which is a holomorphic and non-vanishing function off \( \T \), that is smooth in the entire plane, and satisfies \( D_+(s)=D_-(s)h(s) \) on the unit circle, see \cite[Chapter~I]{Gakhov}. Therefore, the function
\[
\widehat h(s) := D^2(s)h^{\mp1}(s), \quad s\in F_{a^{\pm1}},
\]
is non-vanishing and smooth on the interval \( (a^{-1},a) \) and admits a continuous determination of its logarithm. Thus, the function
\begin{equation}
\label{Sh-2}
S(z) : = \exp\left\{\frac{w(z)}{2\pi\ic}\int_{[a^{-1},a]}\frac{\log\widehat h(s)}{s-z}\frac{\dd s}{w_+(s)}\right\}
\end{equation}
is a holomorphic and non-vanishing function off the interval \( [a^{-1},a] \) with traces satisfying \(S_+(s)S_-(s)=\widehat h(s)\). Given \eqref{varphi}, \eqref{Sh-1}, and \eqref{Sh-2}, let us put
\begin{equation}
\label{calQn-real}
\Qn(z) := \left\{
\begin{array}{rl}
(z/\varphi(z))^nS(z)/D(z), & z\in D_0, \medskip \\ 
\varphi(z)^nD(z)/S(z), & z\in D_\infty,
\end{array}
\right.
\end{equation}
which is a holomorphic and non-vanishing function in \( \C\setminus F \) with a pole of order \( n \) at infinity and
\begin{equation}
\label{calRn-real}
\Rn(z) := \left\{
\begin{array}{rl}
(\varphi(z)/z)^nD(z)/S(z), & z\in D_0, \medskip \\ 
-(1/\varphi(z))^nS(z)/D(z), & z\in D_\infty,
\end{array}
\right.
\end{equation}
which is a holomorphic and non-vanishing function in \( \C\setminus F \) with a zero of multiplicity \( n \) at infinity (one can clearly see that \( \Rn \) is essentially a reciprocal of \( \Qn \) and therefore this definition might appear superfluous, however, for not real \( a \) the relation between these two functions will not be as straightforward and we prefer to present the cases of real and non-real parameters in a uniform fashion). 

\begin{theorem}
\label{thm:main1}
Denote by \( P_n/Q_n \) the two-point Pad\'e approximant of type \( (n,n+1) \)  to \eqref{f} with \( \rho\in \mathcal W_1 \). Set
\begin{equation}
\label{Rn}
R_n(z)  :=  z^{-n}(Q_nf_\rho-P_n)(z), \quad z\in\overline\C\setminus F.
\end{equation}
Then for all \( n \) large enough the polynomial \( Q_n \) has degree \( n \) and can be normalized to be monic. In this case, it holds that
\begin{equation}
\label{asymptotics1}
\left\{
\begin{array}{rll}
Q_n &=& \gamma_n(1+\upsilon_{\Qn}\big)\Qn, \medskip \\
wR_n & = & \gamma_n\big(1+\upsilon_{\Rn}\big)\Rn,
\end{array}
\right.
\end{equation}
locally uniformly in \( \overline\C\setminus F \), where the error rate functions \( \upsilon=\upsilon_{\Qn},\upsilon_{\Rn} \) satisfy
\begin{equation}
\label{error-rate}
|\upsilon(z)| \leq C\left\{
\begin{array}{ll}
1/n, & \rho\in\mathcal W_1, \medskip \\
c^n, & \rho\in\mathcal W_2,
\end{array}
\right.
\end{equation}
for some constants \( C>0 \) and \( c<1 \), and \( \gamma_n \) is a constant such that the limit \( \lim_{z\to\infty}\gamma_n\Qn(z)z^{-n} = 1 \) holds.
\end{theorem}

We can immediately see from \eqref{calQn-real}, \eqref{calRn-real}, and \eqref{asymptotics1} that
\begin{equation}
\label{error}
(f_\rho-P_n/Q_n)(z) = \frac{1+o(1)}{w(z)}\left\{
\begin{array}{rl}
(\varphi^2(z)/z)^n(SD)^2(z), & z\in D_0, \medskip \\ 
-(z/\varphi^2(z))^n/(SD)^2(z), & z\in D_\infty.
\end{array}
\right.
\end{equation}
It further can be deduced from \eqref{equilibrium} that the right hand side of the above equality is geometrically small in \( \overline\C\setminus F \).

\begin{figure}[ht!]
\centering
\includegraphics[scale=.4]{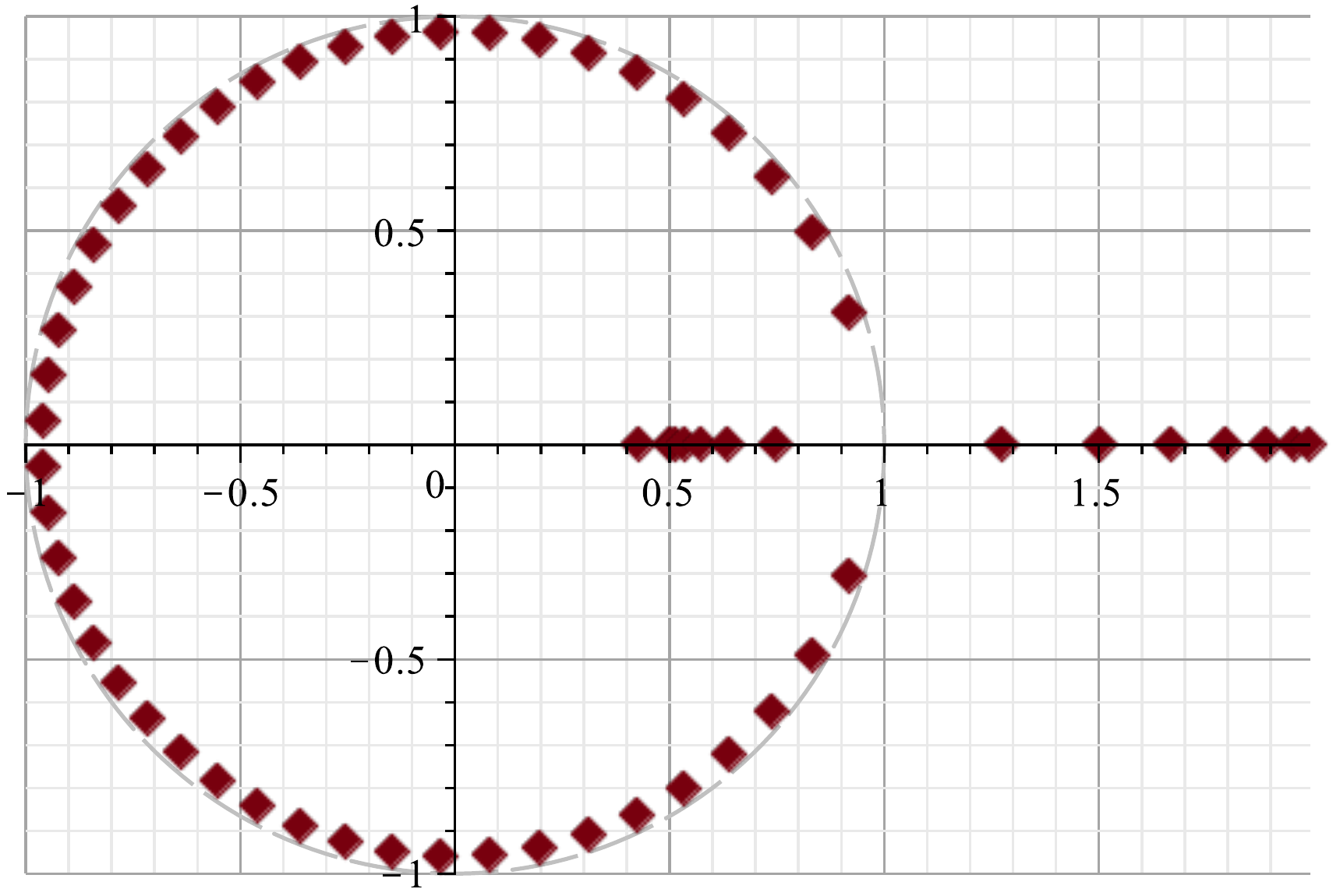}
\caption{\small Zeros of the denominator polynomial \( Q_{60} \) when the approximated pair is given by \( f_0(z) = \log\big(\frac{z-1}{z-1/a}\big) \) and \( f_\infty(z) = \log\big(\frac{z-a}{z-1}\big) \) for \( a=2 \).}
\label{N-real}
\end{figure}

\subsection{Asymptotics of the Approximants: \( a\in\D\setminus(-1,1) \)}

As in the previous subsection, we start by defining the classes of weights \( \rho \) we shall consider. We shall say that \( \rho \in \mathcal W_1 \) if the restriction of \( \rho \) to the arcs \( F_1^\circ,F_{-1}^\circ \) extends to a holomorphic and non-vanishing function around \( F_1,F_{-1} \); there exist constants \( \alpha,\beta>-1 \) such that the restriction of \( \rho \) to \( F_a^\circ,F_{a^{-1}}^\circ \)  is of the form
\begin{equation}
\label{W1-comp}
\widetilde\rho(s)(s-a)^\alpha(s-a^{-1})^\beta
\end{equation}
where \( \widetilde\rho \) is holomorphic and non-vanishing around \( F_a,F_{a^{-1}} \) and the branches of the power functions are holomorphic across \( F_a^\circ,F_{a^{-1}}^\circ \); and
\begin{equation}
\label{sum-cond}
\left\{
\begin{array}{l}
\rho_{|F_a}(z) + \rho_{|F_{-1}}(z)-\rho_{|F_1}(z) \equiv 0, \medskip \\
\rho_{|F_{a^{-1}}}(z) + \rho_{|F_{-1}}(z) - \rho_{|F_1}(z) \equiv 0,
\end{array}
\right.
\end{equation}
in some neighborhood of \( b \) (upper relation) and in some neighborhood of \( b^{-1} \) (lower relation), where, with a slight abuse of notation, we denote by \( \rho_{|F_e}(z) \) not only the restriction of \( \rho \) to \( F_e \), but also its analytic continuation. Unlike the case of the real parameter, the class \( \mathcal W_2 \) will be disjoint from \( \mathcal W_1 \). Let now
\begin{equation}
\label{w}
w(z) = w(z;a) := \sqrt{(z-a)(z-a^{-1})(z-b)(z-b^{-1})}
\end{equation}
be the branch holomorphic in \( \C\setminus(F_{a^{-1}}\cup F_a) \) such that \( w(z) = z^2 + \mathcal O(z) \) as \( z\to\infty \). We shall say that a function \( \rho \) belongs to the class \( \mathcal W_2 \) if
\begin{equation}
\label{h}
h(s) := \rho(s)\left\{
\begin{array}{rl}
w_\pm(s), & s\in F_{a^{\pm1}}, \medskip \\
w(s), & s\in F_{-1}^\circ\cup F_1^\circ,
\end{array}
\right.
\end{equation}
extends to a holomorphic and non-vanishing function in some neighborhood of~\( F \). It is easy to check that when \( h(s)\equiv2 \), we again obtain the pair \( (1/w,-1/w)\).  

The advantage of class \( \mathcal W_2 \) is that the error estimates in the analog of \eqref{asymptotics1} shall be again geometric. However, this class is not very natural as one needs to know the point \( b \) explicitly to define functions in this class, however, finding \( b \) is in general a transcendental problem. The class \( \mathcal W_1 \) is more natural as it for example contains pairs \( (c_0/w_a, c_\infty/w_a) \) for some constants \( c_0\neq\pm c_\infty \), where \( w_a(z) := \sqrt{(z-a)(z-a^{-1})} \) is a branch holomorphic off \( F_{a^{-1}} \cup F_1 \cup F_a \).  

The analogues of the functions \( \varphi \), \( D \), and \( S \) from \eqref{varphi}, \eqref{Sh-1}, and \eqref{Sh-2} are more complicated now as they need to be defined with the help of various differentials on the Riemann surface
\begin{equation}
\label{RS}
\RS := \left\{\z:=(z,w):w^2=(z-a)(z-a^{-1})(z-b)(z-b^{-1})\right\},
\end{equation}
which has genus 1. Moreover, these functions by themselves are not sufficient to define analogs of \( \Qn,\Rn \). Hence, we opt for a different approach. 

The surface \( \RS \) can be realized as two copies of \( \overline\C \) cut along \( F_{a^{-1}}\cup F_a \) and then glued crosswise along the corresponding arcs. We shall denote these copies by \( \RS^{(0)} \) and \( \RS^{(1)} \). Denote by \( \pi:\RS\to\overline\C \) the natural projection \( \pi(\z) =z \). For a point \( z\not\in F_{a^{-1}}\cup F_a \) we also let \( z^{(i)} \) stand for the pull back of \( z \) to \( \RS^{(i)} \). We put 
\[
\boldsymbol\Delta:=\pi^{-1}(F)=\boldsymbol\alpha\cup\boldsymbol\beta\cup\boldsymbol\gamma\cup\boldsymbol\delta,
\]
where \( \boldsymbol\beta :=\pi^{-1}(F_{a^{-1}}) \) is oriented so that \( \RS^{(0)}\setminus\boldsymbol\Delta \) remains on the right when it is traversed in the positive direction, \( \boldsymbol\delta:=\pi^{-1}(F_a) \) is oriented so that \( \RS^{(0)}\setminus\boldsymbol\Delta \) remains on the left when it is traversed in the positive direction, and \( \boldsymbol\alpha:=\pi^{-1}(F_1) \), \( \boldsymbol\gamma:=\pi^{-1}(F_{-1}) \) are oriented so that their positive directions in \( \RS^{(0)} \) coincide with the positive directions of  \( F_1 \), \( F_{-1} \). 

\begin{theorem}
\label{thm:NS}
Let \( h(s) \) be a function on \( F \) for which there exist real constants \( \alpha(e) \), \( e\in E=F\setminus F^\circ \), and the branches of \( (z-e)^{\alpha(e)} \) such that the product
\[
h(s)\prod_{e\in E}(s-e)^{-\alpha(e)}
\]
extends to a non-vanishing H\"older continuous function on \( F \). Then for each \( n\in\N \) there exist a meromorphic  in \( \RS\setminus\boldsymbol\Delta \) function \( \Psi_n \)  and a point \( \z_n \in \RS \) such that 
\begin{itemize}
\item[(i)] \( \Psi_n \) has continuous traces on \( \pi^{-1}(F^\circ) \) that satisfy
\begin{equation}
\label{Psin-jump}
\Psi_{n-}(\s) = \Psi_{n+}(\s)h(s), \quad \s\in\boldsymbol\Delta;
\end{equation}
\item[(ii)]  it has a pole of order \( n\) at \( \infty^{(1)} \) (of order \( n-1 \) if \( \z_n=\infty^{(1)}\)), a simple pole at \( \infty^{(0)} \) (a regular point if \( \z_n=\infty^{(0)}\)), a zero of multiplicity \( n \) at \( 0^{(1)} \) (of multiplicity \(n+1\) if \( \z_n=0^{(1)}\)), a simple zero at \( \z_n\not\in\boldsymbol\Delta\cup\{\infty^{(0)},\infty^{(1)},0^{(1)}\} \) (when $\z_n\in\boldsymbol\Delta$ we use explicit representation \eqref{Psin} for $\Psi_n$ to treat $\z_n$ as a zero for both $\Psi_{n+}$ and $\Psi_{n-}$), and otherwise is non-vanishing and finite;
\item[(iii)] it holds that
\begin{equation}
\label{psin-endp}
\big|\Psi_n\big(z^{(k)}\big)\big|^2 \sim \left\{
\begin{array}{ll}
|z-e|^{(-1)^{1-k}\alpha(e)+m_n(e)} & \text{as} \quad z\to e\in\{a,a^{-1}\}, \medskip \\
|z-e|^{(-1)^{1-k}\alpha(e)+m_n(e)} & \text{as} \quad D_0\ni z\to e\in\{b,b^{-1}\}, \medskip \\
|z-e|^{(-1)^k\alpha(e)+m_n(e)} & \text{as} \quad D_\infty\ni z\to e\in\{b,b^{-1}\},
\end{array}
\right.
\end{equation}
for \( k\in\{0,1\} \), where \( m_n(e):=1 \) when \( \pi(\z_n)=e \) and \( m_n(e):=0 \) otherwise\footnote{The notation \( |g_1(z)|\sim|g_2(z)| \) as \( z\to z_0 \) means that there exists a constant \( C>1 \) such that \( C^{-1} \leq |(g_1/g_2)(z)| \leq C \) in some neighborhood of \( z_0 \).}.
\end{itemize}

Conversely, if \( \Psi \) is a function with  a zero of multiplicity at least \( n \) at \( 0^{(1)} \), a pole of order at most \( n \) at \( \infty^{(1)} \), at most a simple pole at \( \infty^{(0)} \), and no other poles, and if it satisfies \eqref{Psin-jump} and 
\[
\big|\Psi_n\big(z^{(k)}\big)\big|^2 = \left\{
\begin{array}{ll}
\mathcal O\big(|z-e|^{(-1)^{1-k}\alpha(e)}\big) & \text{as} \quad z\to e\in\{a,a^{-1}\}, \medskip \\
\mathcal O\big(|z-e|^{(-1)^{1-k}\alpha(e)}\big) & \text{as} \quad D_0\ni z\to e\in\{b,b^{-1}\}, \medskip \\
\mathcal O\big(|z-e|^{(-1)^k\alpha(e)}\big) & \text{as} \quad D_\infty\ni z\to e\in\{b,b^{-1}\},
\end{array}
\right.
\]
then \( \Psi \) a constant multiple of~\( \Psi_n \).
\end{theorem}

As mentioned before, the functions \( \Psi_n \) can be explicitly expressed via various differential on \( \RS \) as well as Riemann's theta functions, see \eqref{Psin} further below.

In our asymptotic analysis we shall be interested only in the indices \( n \) for which points \( \z_n \) stay away from \( \infty^{(1)} \). As stated in the following proposition, there are infinitely many such indices.

\begin{proposition}
\label{prop:Ne}
Given \( \varepsilon>0 \), let \( D_\varepsilon(z) \) be a disk of radius \( \varepsilon \) in the spherical metric around \( z\in\overline\C \) and \( D_\varepsilon(\z) \) be the connected component of \( \pi^{-1}(D_\varepsilon(z)) \) containing \( \z \). Define
\begin{equation}
\label{Ne}
\N_\varepsilon :=\left\{n\in\N: \z_n\not\in D_\varepsilon\big(\infty^{(1)} \big)\right\}.
 \end{equation}
 Then for all \( \varepsilon \) small enough either \( n \) or \( n-1 \) belongs to \( \N_\varepsilon \).
\end{proposition}

In accordance with our notation, let us put \( B^{(i)}:= \pi^{-1}(B)\cap \RS^{(i)} \) for any set \( B \). Now we are ready to define the analogues of \eqref{calQn-real} and \eqref{calRn-real} for non-real \( a \). Set
\begin{equation}
\label{calQn}
\Qn(z) := \left\{
\begin{array}{rl}
\Psi_{n|D_0^{(0)}}(\z), & z\in D_0, \medskip \\
\Psi_{n|D_\infty^{(1)}}(\z), & z\in D_\infty,
\end{array}  \right.
\end{equation}
which is a sectionally holomorphic function in \( \C\setminus F \) with a pole of order \( n \) at infinity, and put
\begin{equation}
\label{calRn}
\Rn(z) :=  \frac1{z^n}\left\{
\begin{array}{rl}
\Psi_{n|D_0^{(1)}}(\z), & z\in D_0, \medskip \\
-\Psi_{n|D_\infty^{(0)}}(\z), & z\in D_\infty,
\end{array}
\right.
\end{equation}
which is also a sectionally holomorphic function in \( \overline\C\setminus F \)  with a zero of multiplicity \( n-1 \) at infinity\footnote{Again, these orders might change depending on the location of \( \z_n \).}. Due to the specifics of the Riemann-Hilbert analysis, which is used to study the behavior of the Pad\'e approximants, we shall also need the following functions. Let \( \Upsilon_n \) be a rational function on \( \RS \) that is finite except for two simple poles at \( \infty^{(0)} \) and \( \z_n \), and has a simple zero at \( 0^{(1)} \) (such a function is unique up to a scalar factor). Set 
\[
\Psi_n^\star := \Psi_n \Upsilon_n
\]
and define \( \Qn^\star \) and \( \Rn^\star \) via \eqref{calQn} and \eqref{calRn}, respectively, with \( \Psi_n \) replaced by \( \Psi_n^\star \) and \( z^n \) replaced by \( z^{n+1} \) in \eqref{calRn}. These functions are holomorphic in \( \C\setminus F \), \( \Qn^\star \) has a pole of order \( n \) at infinity while \( \Rn^\star \) has a zero of multiplicity \( n-1 \) there.

\begin{theorem}
\label{thm:main2}
Denote by \( P_n/Q_n \) the two-point Pad\'e approximant of type \( (n,n+1) \)  to \eqref{f} with \( \rho\in \mathcal W_1 \cup \mathcal W_2 \) and let \( R_n \) be given by \eqref{Rn}. Further, let \( \Qn \) and \( \Rn \) be given by \eqref{calQn} and \eqref{calRn} for \( \Psi_n \) defined as in Theorem~\ref{thm:NS} with \( h \) given by \eqref{h}. Then for any \( \varepsilon> 0 \) and all \( n\in\N_\varepsilon \) large enough the polynomial \( Q_n \) has degree \( n \) and can be normalized to be monic. In this case, it holds that
\begin{equation}
\label{asymptotics2}
\left\{
\begin{array}{rll}
Q_n &=& \gamma_n\left[\big(1+\upsilon_{n1}\big)\Qn + \upsilon_{n2}\mathcal Q_{n-1}^\star\right], \medskip \\
wR_n & = & \gamma_n\left[\big(1+\upsilon_{n1}\big)\Rn + \upsilon_{n2}\mathcal R_{n-1}^\star\right],
\end{array}
\right.
\end{equation}
locally uniformly in \( \overline\C\setminus F \), where \( \gamma_n \) is a constant such that \( \lim_{z\to\infty}\gamma_n\Qn(z)z^{-n} = 1 \) and the functions \( \upsilon=\upsilon_{nj} \) satisfy  \eqref{error-rate}.
\end{theorem}

Similarly to \eqref{error} we have that
\[
f_\rho-\frac{P_n}{Q_n} = \frac{z^n\Rn}{w\Qn} \frac{1+\upsilon_{n,1}+\upsilon_{n,2}(\mathcal R_{n-1}^\star/\Rn)}{1+\upsilon_{n,1}+\upsilon_{n,2}(\mathcal Q_{n-1}^\star/\Qn)}.
\]
Due to the presence of a floating zero \( \z_n \), the above formula does not immediately imply the locally uniform convergence of the approximants to \( f_\rho \). Indeed, when 
\[
\z_n\in D_{\mathcal R} := D_0^{(1)} \cup D_\infty^{(0)},
\]
it holds that \( \Rn(z_n)=0 \), which yields that the approximant has an additional interpolation point near \( z_n \). However, when
\[
\z_n\in D_{\mathcal Q} := D_0^{(0)} \cup D_\infty^{(1)},
\]
it holds that \( \Qn(z_n)=0 \) and therefore the approximant has  pole in a vicinity of  \( z_n \). The following results help us further elucidate the situation.

\begin{theorem}
\label{thm:NS2}
The functions \( \Psi_n \) can be normalized so that for any closed set \( B \subset \overline\C\setminus F \) and any \( \delta>0 \) there exist positive constants \( C(B) \) and  \( C_\delta(B) \) such that
\begin{equation}
\label{calQn-bound}
|\Qn(z)|e^{-ng(z)}\left\{
\begin{array}{ll}
\leq C(B), & z\in B, \medskip \\
\geq C_\delta(B), & z\in B\setminus \pi\big(D_{\mathcal Q}\cap D_\delta(\z_n)\big),
\end{array}
\right.
\end{equation}
where \( g(z) \) is a continuous function in \( \C \) that is harmonic in \( \C\setminus F \) and satisfies 
\begin{equation}
\label{fun-g}
g(z) = \log|z| + \mathcal O(1), \quad  z\to\infty, \quad \text{and} \quad 2g(s) = \log|s|, \quad s\in F
\end{equation}
(it follows from the minimum principle for superharmonic functions that \( 2g(z)-\log|z|>0 \) in \( \C\setminus F \)). Moreover, the constants \( C(B) \) and  \( C_\delta(B) \) can be adjusted so that
\begin{equation}
\label{calRn-bound}
|\Rn(z)|e^{(n-1)g(z)}\left\{
\begin{array}{ll}
\leq C(B), & z\in B, \medskip \\
\geq C_\delta(B), & z\in B\setminus \pi\big(D_{\mathcal R}\cap D_\delta(\z_n)\big).
\end{array}
\right.
\end{equation}
The functions \( \Upsilon_n \) can be normalized so that inequalities \eqref{calQn-bound} and \eqref{calRn-bound} hold for \( \Qn^\star \) and \( \Rn^\star \) as well. Moreover, for any \( \delta>0 \) there exists a constant \( C_\delta \) such that
\begin{equation}
\label{compare}
C_\delta \geq \left\{
\begin{array}{ll}
|\mathcal Q_{n-1}^\star/\Qn| & \text{in} \quad \overline\C\setminus \pi\big(\overline D_{\mathcal Q}\cap D_\delta(\z_n) \big), \medskip \\
|\mathcal R_{n-1}^\star/\Rn| & \text{in} \quad \overline\C\setminus \left(D_\delta(\infty) \cup \pi\big(\overline D_{\mathcal R}\cap D_\delta(\z_n) \big)\right).
\end{array} 
\right.
\end{equation}
\end{theorem}

In view of Buslaev's theorem, it should be clear that \( 2g(z)-\log|z|=g_F(z,0)+g_F(z,\infty) \). 

\begin{figure}[ht!]
\centering
\subfigure[]{\includegraphics[scale=.3]{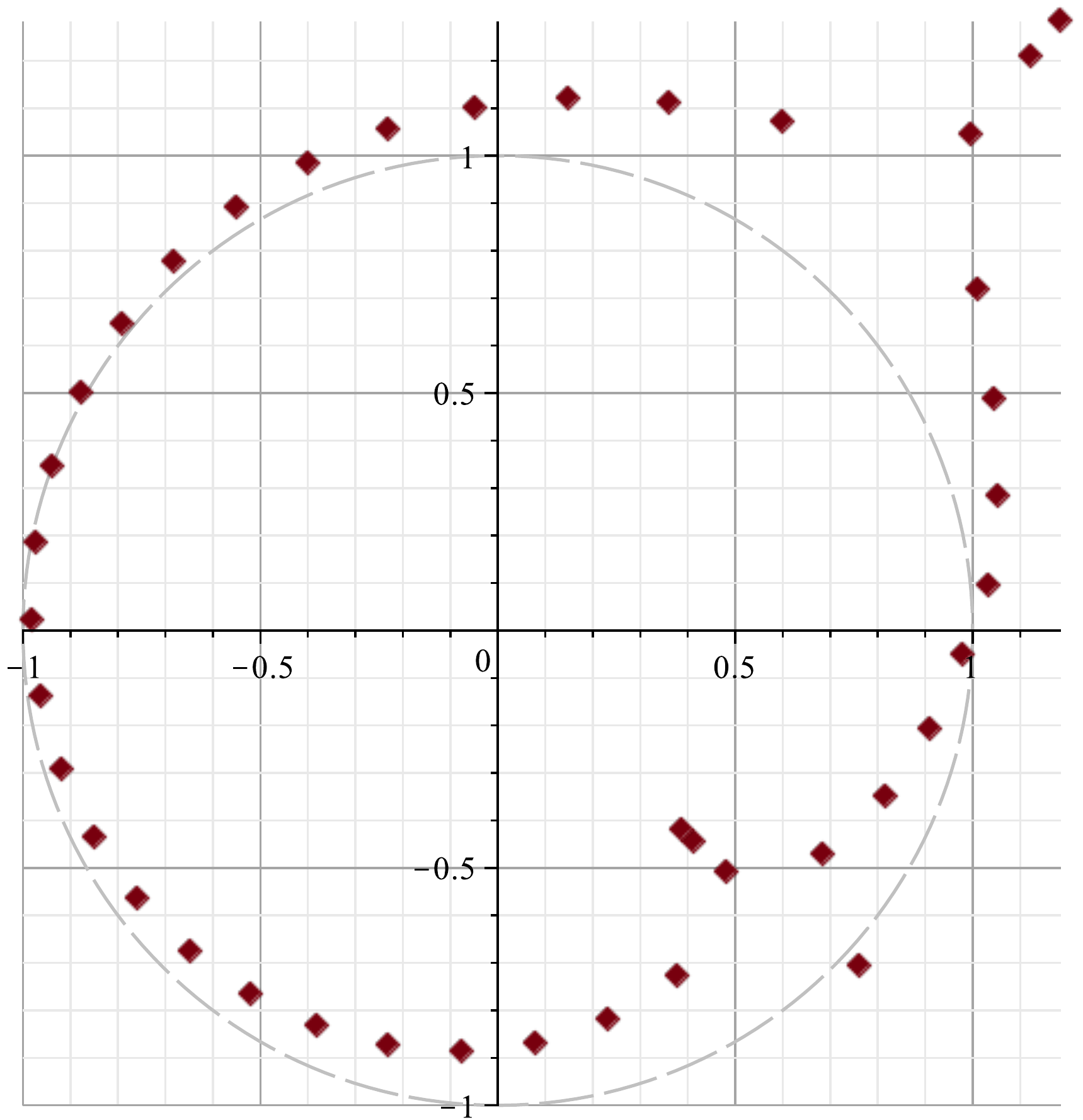}}
\quad\quad
\subfigure[]{\includegraphics[scale=.3]{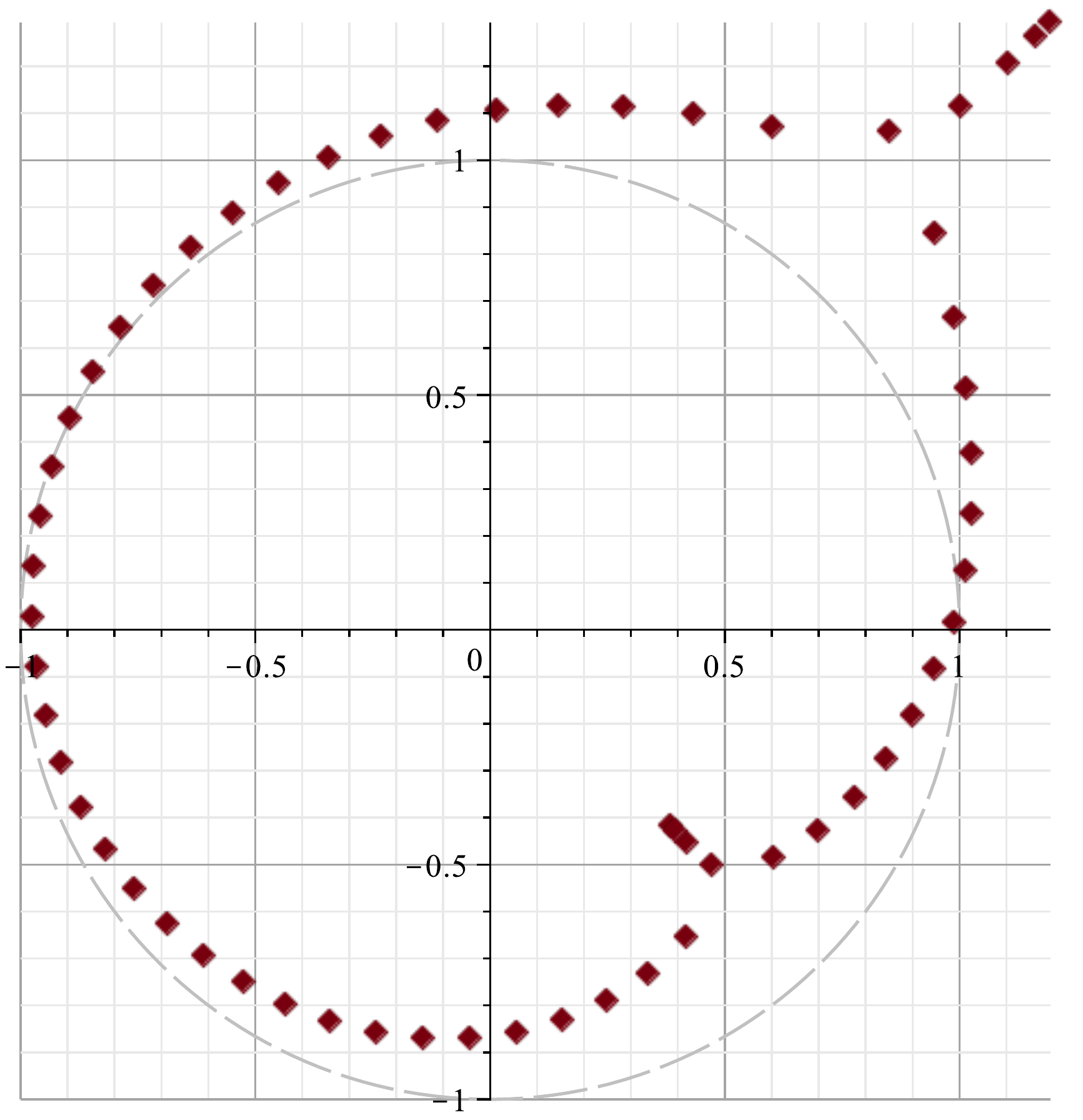}}
\caption{\small Zeros of the denominator polynomial (a) \( Q_{40} \) and (b) \( Q_{60} \) when the approximated pair is given by \( f_0(z) = \log\big(\frac{z-1}{z-1/a}\big) \) and \( f_\infty(z) = \log\big(\frac{z-a}{z-1}\big) \) for \( a = 1.2+1.3\ic \). One can clearly see one zero of \( Q_{40} \) not being aligned along Buslaev's compact \( F \).}
\label{N-complex}
\end{figure}

The author would like to thank Andrei Mart\'{\i}nez Finkelshtein for many valuable discussions.

\section{Proof of Theorems~\ref{thm:NS},~\ref{thm:NS2} and Proposition~\ref{prop:Ne}}
\label{sec:NS}

Let \( \RS \) be the Riemann surface defined in \eqref{RS}. We consider each \( \RS^{(i)} \) to be closed subsets of \( \RS \), i.e., it does contain cycles \( \boldsymbol\beta,\boldsymbol\delta \).  We define the conformal involution on \( \RS \) by \( \z=(z,w) \mapsto \z^* =(z,-w) \). It is easy to see that the pair \( (\boldsymbol\alpha,\boldsymbol\beta) \) forms a homology basis on \( \RS \). In particular, \( \RS_{\boldsymbol\alpha,\boldsymbol\beta}:= \RS\setminus(\boldsymbol\alpha\cup\boldsymbol\beta) \) is simply connected.

\subsection{Nuttall's Differential}
\label{ssec:ND}

Let \( w(\z) := (-1)^iw(z) \), \( \z\in \RS^{(i)} \), where \( w(z) \) is the branch defined in \eqref{w}. For convenience, set
\[
v(\z) := \frac{(z-b)(z-b^{-1})}{w(\z)}
\]
and \( v(z):= v(\z) \) for \( \z\in\RS^{(0)}\setminus\{\boldsymbol\beta\cup\boldsymbol\gamma\cup\boldsymbol\delta\} \). Notice that \( v(0)=v(\infty)=1 \). The differential
\begin{equation}
\label{Nuttall}
\mathcal N(\s) = \frac{1-v(\s)}{2s}\dd s
\end{equation}
is holomorphic except for two simple poles at \( 0^{(1)} \) and \( \infty^{(1)} \) with respective residues \( 1 \) and \( -1 \). Moreover, one can readily check using \eqref{diff} that all the periods of \( \mathcal N(\s) \) are purely imaginary and therefore we can define
\begin{equation}
\label{periods}
\omega := -\frac1{2\pi\ic}\oint_{\boldsymbol\beta}\mathcal N \quad \text{and} \quad \tau:=\frac1{2\pi\ic}\oint_{\boldsymbol\alpha}\mathcal N,
\end{equation} 
which are clearly real constants.

\begin{lemma}
\label{lem:Phi}
Let $\sqrt a$ be the principal value of the square root. Define
\begin{equation}
\label{Phi}
\Phi(\z) := \sqrt a \exp\left\{\int_{\boldsymbol a}^{\z}\mathcal N\right\}, \quad \z\in\RS_{\boldsymbol\alpha,\boldsymbol\beta},
\end{equation}
where the path of integration belongs entirely to \( \RS_{\boldsymbol\alpha,\boldsymbol\beta} \).  The function \( \Phi(\z) \) is holomorphic and non-vanishing in \( \RS_{\boldsymbol\alpha,\boldsymbol\beta}\setminus\big\{\infty^{(1)},0^{(1)}\big\} \) with a simple pole at \( \infty^{(1)} \) and a simple zero at \( 0^{(1)}\in \RS^{(1)} \). It holds that \( \Phi(\z)\Phi(\z^*)=z \) and the traces of \( \Phi(\z) \) satisfy
\begin{equation}
\label{Phi-jumps}
\Phi_+ = \Phi_- \left\{
\begin{array}{lll}
\exp\left\{2\pi\ic\tau\right\} & \text{on} & \boldsymbol\beta, \medskip \\
\exp\left\{2\pi\ic\omega\right\} & \text{on} & \boldsymbol\alpha.
\end{array}
\right.
\end{equation}
Moreover, $|\Phi(\z)|^2<|z|$ for $z\in D_{\mathcal R} \) and $|\Phi(\z)|^2>|z|$ for $z\in D_{\mathcal Q} \).
\end{lemma} 

\begin{proof}
The holomorphy properties of \( \Phi \) follow immediately from the corresponding properties of \( \mathcal N \). Since \( v(\z) = -v(z) \)  for \( \z\in\RS^{(1)}\setminus\{\boldsymbol\beta\cup\boldsymbol\gamma\} \), it holds that
\begin{equation}
\label{symmetry}
\Phi(\z)\Phi(\z^*) = a\exp\left\{\int_a^z\frac{(1-v(t))\dd t}{2t}+\int_a^z\frac{(1+v(t))\dd t}{2t}\right\} = a\exp\left\{\int_a^z\frac{\dd t}{t}\right\} = z.
\end{equation}
Furthermore, we get on \( \boldsymbol \alpha \) and \( \boldsymbol\beta \) that
\[
\Phi_+ = \Phi_-\exp\left\{-\oint_{\boldsymbol\beta}\mathcal N\right\} \quad \text{and} \quad \Phi_+ = \Phi_-\exp\left\{\oint_{\boldsymbol\alpha}\mathcal N\right\},
\]
respectively, which yields \eqref{Phi-jumps}, see \eqref{periods}. Finally, recall that $F$ consists of the critical trajectories of the quadratic differential $-(v(z)\mathrm dz/z)^2$, see \eqref{diff}. Hence, the integral of $v(t)\dd t/t$ on any subarc of \( F \) is purely imaginary and therefore
\[
\big|\Phi^2(\s)\big| = |s|\exp\left\{\pm\re\left(\int_a^s\frac{v(t)}t\mathrm dt\right)\right\} = |s|, \quad \s\in\boldsymbol\Delta,
\]
where the sign \( - \) is used if \( \s\in\RS^{(0)} \) and the sign \( + \) is used if \( \s\in\RS^{(1)} \). The last conclusion of the lemma now follows from the maximum modulus principle.
\end{proof}

\subsection{Holomorphic Differentials}

It can be readily checked that
\[
\mathcal H(\s) := \frac{C\dd s}{w(\s)}, \quad C:=\left(\oint_{\boldsymbol\alpha}\frac{\dd s}{w(\s)}\right)^{-1},  \quad \mathbf B:= \oint_{\boldsymbol \beta}\mathcal H,
\]
is a holomorphic differential on \( \RS \) (unique up to a multiplicative constant). It is also known that \( \im(\mathbf B) >0 \). The proof of the following lemma is absolutely analogous to the proof of Lemma~\ref{lem:Phi}.

\begin{lemma}
\label{lem:Asigma}
Given a constant \( \sigma\in\C \), define
\begin{equation}
\label{Asigma}
A_\sigma(z) := \exp\left\{-2\pi\ic\sigma\int_{\boldsymbol a}^{\z}\mathcal H\right\}, \quad \z\in\RS_{\boldsymbol\alpha,\boldsymbol\beta},
\end{equation}
where the path of integration belongs entirely to \( \RS_{\boldsymbol\alpha,\boldsymbol\beta} \).  The function \( A_\sigma(\z) \) is holomorphic and non-vanishing in \(  \RS_{\boldsymbol\alpha,\boldsymbol\beta} \). It holds that \(  A_\sigma(\z) A_\sigma(\z^*)\equiv1 \) and the traces of \( A_\sigma(\z) \) on \( \boldsymbol\alpha,\boldsymbol\beta \) satisfy
\begin{equation}
\label{Asigma-jumps}
A_{\sigma+} = A_{\sigma-} \left\{
\begin{array}{rll}
\exp\left\{-2\pi\ic\sigma\right\} & \text{on} & \boldsymbol\beta, \medskip \\
\exp\left\{2\pi\ic\mathbf B\sigma\right\} & \text{on} & \boldsymbol\alpha.
\end{array}
\right.
\end{equation}
\end{lemma}

\subsection{Cauchy's Differential}

Denote by \( \mathcal C_{\z} \) the unique meromorphic differential that has two  simple poles at \( \z \) and \( \z^* \) with residues \( 1 \) and \( -1 \), respectively, and whose \( \boldsymbol\alpha \)-period is zero. When \( \pi(\z)\in\C \), one can readily check that
\begin{equation}
\label{CK}
\mathcal C_{\z}(\s) = \frac{w(\z)}{s-z}\frac{\dd s}{w(\s)} - \left(\oint_{\boldsymbol\alpha}\frac{w(\z)}{s-z}\frac{\dd s}{w(\s)}\right)\mathcal H(\s).
\end{equation}

\begin{lemma}
\label{lem:Sh}
Let \( h(z) \) be as in Theorem~\ref{thm:NS}. Fix a smooth determination of 
\[
\log h(s) - \sum_{e\in E}\alpha(e)\log(z-e)
\]
on each of the arcs \( F_a,F_{a^{-1}},F_1,F_{-1} \). Define \( \lambda_h(\s) := -\log h(s) \), \( s\in F^\circ \), and
\begin{equation}
\label{Sh}
S_h(\z) := \exp\left\{\frac1{4\pi\ic}\oint_{\boldsymbol\Delta}\lambda_h\mathcal C_{\z}\right\}, \quad \z\in\RS\setminus\boldsymbol\Delta.
\end{equation}
The function \( S_h(\z) \) is holomorphic and non-vanishing in \( \RS\setminus\boldsymbol\Delta \). It holds that \( S_h(\z)S_h(\z^*)\equiv1 \) and the traces of \( S_h(\z) \) satisfy
\begin{equation}
\label{Sh-jumps}
S_{h+}(\s) =  \frac{S_{h-}(\s)}{h(s)}\left\{
\begin{array}{rcl}
1, & \text{on} & \boldsymbol\Delta\setminus\boldsymbol\alpha, \medskip \\
\exp\big\{-\oint_{\boldsymbol\Delta}\lambda_h\mathcal H\big\}, & \text{on} & \boldsymbol\alpha,
\end{array}
\right.
\end{equation}
where the points of self-intersection need to be excluded. Moreover, it holds that
\begin{equation}
\label{Sh-endp}
\big|S_h\big(z^{(k)}\big)\big|^2 \sim \left\{
\begin{array}{ll}
|z-e|^{(-1)^{1-k}\alpha(e)} & \text{as} \quad z\to e\in\{a,a^{-1}\}, \medskip \\
|z-e|^{(-1)^{1-k}\alpha(e)} & \text{as} \quad D_0\ni z\to e\in\{b,b^{-1}\}, \medskip \\
|z-e|^{(-1)^k\alpha(e)} & \text{as} \quad D_\infty\ni z\to e\in\{b,b^{-1}\},
\end{array}
\right. \quad k\in\{0,1\}.
\end{equation}
\end{lemma}
\begin{proof}
Let \( \boldsymbol\nu \) be an involution-symmetric cycle on \( \RS \) passing through ramification points \( \p_1 ,\p_2 \) and \( \lambda_{\boldsymbol\nu} \) be an involution-symmetric function on \(  \boldsymbol\nu \) such that
\[
\lambda_{\boldsymbol\nu}(\z) + \alpha_1\log(z-p_1) + \alpha_2\log(z-p_2)
\]
is H\"older smooth on \( \boldsymbol\nu \) for some real constants \( \alpha_1,\alpha_2 \), where the determinations of the logarithms are holomorphic across \( \pi(\boldsymbol\nu)\setminus\{p_1,p_2\} \). Set
\[
\Lambda_{\boldsymbol\nu}(\z) := \frac1{4\pi\ic}\oint_{\boldsymbol\nu}\lambda_{\boldsymbol\nu}\mathcal C_{\z}.
\]
Since \( \mathcal C_{\z^*} = - \mathcal C_{\z} \), it holds that \( \Lambda_{\boldsymbol\nu}(\z) + \Lambda_{\boldsymbol\nu}(\z^*) \equiv 0 \). Moreover, it is known \cite[Eq.~(2.7)--(2.9)]{Zver71} that \( \Lambda_{\boldsymbol\nu}(\z) \) is a holomorphic function in \( \RS\setminus(\boldsymbol\nu\cup\boldsymbol\alpha) \) with continuous traces on \( (\boldsymbol\nu\setminus\boldsymbol\alpha)\setminus\{\p_1,\p_2\} \) and \( (\boldsymbol\alpha\setminus\boldsymbol\nu)\setminus\{\p_1,\p_2\} \) that satisfy
\[
\Lambda_{\boldsymbol\nu+}(\s) - \Lambda_{\boldsymbol\nu-}(\s) = \left\{
\begin{array}{rl}
\lambda_{\boldsymbol\nu}(\s), & \s\in\boldsymbol\nu\setminus\boldsymbol\alpha, \medskip \\
-\oint_{\boldsymbol\nu}\lambda_{\boldsymbol\nu}\mathcal H, & \s\in\boldsymbol\alpha\setminus\boldsymbol\nu,
\end{array}
\right.
\]
where we used the fact that \( \lambda_{\boldsymbol\nu}(\s) = \lambda_{\boldsymbol\nu}(\s^*) \) and the jumps need to be added up on subarcs of \( \boldsymbol\nu \cap \boldsymbol\alpha \).  In the absence of logarithmic singularities, i.e., when all \( \alpha(e) =0 \), the claims of the lemma now follow by summing up \( \Lambda_{\boldsymbol\nu} \) over all \( \boldsymbol\nu\in\{\boldsymbol\alpha,\boldsymbol\beta,\boldsymbol\gamma,\boldsymbol\delta\} \) while taking \( \lambda_{\boldsymbol\nu} := \lambda_{h|\boldsymbol\nu} \).

Let \( \Lambda :=\sum_{\boldsymbol\nu=\boldsymbol\alpha,\boldsymbol\beta,\boldsymbol\gamma,\boldsymbol\delta}\Lambda_{\boldsymbol\nu}\). When a branch point \( \boldsymbol p \) is such that \( \pi(\boldsymbol p)\in\{a,a^{-1}\} \), there is exactly one cycle from the chain \( \boldsymbol\Delta \) passing through \( \boldsymbol p \) (either \( \boldsymbol\beta \) or \( \boldsymbol\delta \)). Moreover, since \( \boldsymbol\beta \) and \( \boldsymbol\delta \) separate \( \RS \) into the sheets \( \RS^{(0)}\setminus\{\boldsymbol\beta \cup \boldsymbol\delta\} \) and \( \RS^{(1)}\setminus\{\boldsymbol\beta \cup \boldsymbol\delta\} \), the analysis of \cite[Section~5.2]{ApY15} applies and yields that
\[
\Lambda\big(z^{(k)}\big) = (-1)^{1-k}\frac{\alpha(p)}{2}\log(p-z)+\mathcal O(1) \quad \text{as} \quad z\to p\in\big\{a,a^{-1}\big\},
\]
for \( k\in\{0,1\} \), where \( \log(p-\cdot) \) is holomorphic in some neighborhood of \( p=a \) cut along \( F_p \). The situation when \( p\in\{b,b^{-1}\} \) is again very similar to the one discussed in \cite[Section~5.2]{ApY15}. Clearly, the singular behavior around \( p \) comes from the first term in \eqref{CK}. As explained in \cite[Section~I.8.5]{Gakhov}, to understand this behavior it is enough to find a function that has logarithmic singularity at \( p \) and the same jumps across the cycles comprising \( \boldsymbol\Delta \). Thus, it can be checked that
\[
\Lambda\big(z^{(k)}\big) = \pm(-1)^{1-k}\frac{\alpha(p)}{2}\log(z-p)+\mathcal O(1) \quad \text{as} \quad z\to p\in\big\{b,b^{-1}\big\},
\]
where the sign \( + \) is used if \( z\in D_0 \) and the sign \( - \) is used if \( z\in D_\infty \), and \( \log(\cdot-p) \) has a jump along \( F_a \) when \( p=b \) and \( F_{a^{-1}} \) when \( p=b^{-1} \).
\end{proof}

\subsection{Jacobi Inversion Problem}

We define Abel's map on \( \RS \) by
\[
\am(\z) := \int_{\boldsymbol a}^{\z}\mathcal H, \quad \z \in \RS_{\boldsymbol\alpha,\boldsymbol\beta},
\]
where the path of integration lies entirely in \( \RS_{\boldsymbol\alpha,\boldsymbol\beta} \), and set \( \am(\z) := \am_+(\z) \) when \( \z\in\boldsymbol\alpha \cup \boldsymbol\beta \). Since \( \RS \) has genus \( 1 \), any Jacobi inversion problem is uniquely solvable on \( \RS \). In particular, given a function \( h \) and an integer \( n\in\N \), there exist unique \( \z_n=\z_n(h)\in\RS \) and \( j_n,m_n\in\Z \) such that
\begin{equation}
\label{jip}
\am(\z_n) = \am\big(\infty^{(0)}\big) -\frac1{2\pi\ic}\oint_{\boldsymbol\Delta}\lambda_h\mathcal H + n(\omega+\mathbf B\tau) + j_n + \mathbf Bm_n. 
\end{equation}

\begin{lemma}
\label{lem:JIP}
For \( \varepsilon>0 \), let \( \N_\varepsilon \) be defined by \eqref{Ne}.  Then the conclusions of Proposition~\ref{prop:Ne} hold true. Further, let \( \z_n^\star=\z_n^\star(h) \) be the solution of the following Jacobi inversion problem:
\[
\am\big(\z_n^\star+0^{(1)}-\z_n-\infty^{(0)}\big) \in \Z+\mathbf B\Z. 
\]
Then there exists a domain \( U_\varepsilon\ni \infty^{(0)}  \) such that \( \z_{n-1}^\star\not\in U_\varepsilon \) for all  \( n\in\N_\varepsilon \).
\end{lemma}
\begin{proof}
According to Riemann's relations, it holds that
\[
\am\big( \infty^{(1)} \big) -\am\big( 0^{(1)} \big) = \int_{0^{(1)}}^{\infty^{(1)}}\mathcal H = \frac1{2\pi\ic}\oint_{\boldsymbol\beta}\mathcal M_{\infty^{(1)},0^{(1)}},
\]
where \( \mathcal M_{\z_1,\z_2} \) is a meromorphic differential having two simple poles at \( \z_1 \) and \( \z_2 \) with residues \( 1 \) and \( -1 \), respectively, and zero period on \( \boldsymbol\alpha \). In fact,
\[
\mathcal M_{\infty^{(1)},0^{(1)}} = - \mathcal N + 2\pi\ic\tau\mathcal H
\]
as one can see from \eqref{periods}. That is, it holds that
\[
\am\big( \infty^{(1)} \big) -\am\big( 0^{(1)} \big) = -\frac1{2\pi\ic}\oint_{\boldsymbol\beta}\mathcal N + \tau\oint_{\boldsymbol\beta}\mathcal H = \omega+\mathbf B\tau.
\]
It also follows from \eqref{jip} that
\[
\am(\z_n ) -\am(\z_{n-1}) - (\omega+\mathbf B\tau) \in \Z+\mathbf B\Z.
\]
The continuity of \( \am(\z) \) and the unique solvability of the Jacobi inversion problem now yield that if \( \z_n\to\infty^{(1)} \) along some subsequence \( \N^\prime \), then \( \z_{n-1}\to 0^{(1)} \) as \( \N^\prime\ni n\to\infty \), which proves unboundedness of \( \N_\varepsilon \) as well as the fact that either \( n \) or \( n-1 \) is in \( \N_\varepsilon \) for all \( \varepsilon \) small enough.  The same argument proves the last claim of the lemma since \( \z_{n-1}^\star \to \infty^{(0)} \) along some subsequence implies that \( \z_{n-1} \to 0^{(1)} \) and respectively \( \z_n\to\infty^{(1)} \) along the same subsequence. 
\end{proof}

\subsection{Riemann's Theta Function}

Recall that the theta function associated with $\mathbf B$ is an entire transcendental function defined by
\[
\theta(u) := \sum_{n\in\Z}\exp\bigg\{\pi\ic\mathbf Bn^2 + 2\pi\ic un\bigg\}, \quad u\in\C.
\]

\begin{lemma}
\label{lem:Thetan}
Let \( h \) and \( \z_n \) be as above. Define
\begin{equation}
\label{Thetan}
\Theta_n(\z) := \frac{\theta\left(\am(\z) - \am(\z_n) - \frac{1+\mathbf B}2\right)}{\theta\left(\am(\z) - \am\big(\infty^{(0)}\big) - \frac{1+\mathbf B}2\right)}.
\end{equation}
The function $\Theta_n$ is meromorphic in $\RS_{\boldsymbol\alpha,\boldsymbol\beta}$ with a simple zero at \( \z_n \), a simple pole at $\infty^{(0)}$, and otherwise non-vanishing and finite. In fact, it is holomorphic across $\boldsymbol\beta$ and
\begin{equation}
\label{Thetan-jump}
\Theta_{n+} = \Theta_{n-} \exp\left\{\oint_{\boldsymbol\Delta}\lambda_h\mathcal H-2\pi\ic n(\tau+\mathbf B\omega)-2\pi\ic\mathbf Bm_n\right\} \quad \text{on} \quad \boldsymbol\alpha.
\end{equation}
\end{lemma}
\begin{proof}
It can be directly checked that
\begin{equation}
\label{theta-periods}
\theta(u + j + \mathbf Bm) = \exp\big\{-\pi\ic\mathbf Bm^2 - 2\pi\ic um\big\}\theta(u), \quad j,m\in\Z.
\end{equation}
Moreover, it is known that \( \theta(u)=0 \) if and only if \( u  = (1+\mathbf B)/2  + j+\mathbf Bm \), \(j,m\in\Z\). Further, recall that
\[
\am_+ - \am_- = \left\{
\begin{array}{rcl}
-\oint_{\boldsymbol\beta}\mathcal H & \text{on} & \boldsymbol\alpha, \medskip \\
\oint_{\boldsymbol\alpha}\mathcal H & \text{on} & \boldsymbol\beta,
\end{array}
\right. = \left\{
\begin{array}{rcl}
-\mathbf B & \text{on} & \boldsymbol\alpha, \medskip \\
1 & \text{on} & \boldsymbol\beta.
\end{array}
\right.
\]
Therefore, it holds for \( \s\in\boldsymbol\beta \) that
\[
\Theta_{n+}(\s) = \frac{\theta\left(\am_+(\z) - \am(\z_n) - \frac{1+\mathbf B}2\right)}{\theta\left(\am_+(\z) - \am\big(\infty^{(0)}\big) - \frac{1+\mathbf B}2\right)} = \frac{\theta\left(1+ \am_-(\z) - \am(\z_n) - \frac{1+\mathbf B}2\right)}{\theta\left(1 + \am_-(\z) - \am\big(\infty^{(0)}\big) - \frac{1+\mathbf B}2\right)} = \Theta_{n-}(\s).
\]
That is, \( \Theta_n \) is holomorphic across \( \boldsymbol\beta \) as claimed. Similarly, we get on \( \boldsymbol\alpha \) that
\[
\Theta_{n+}  = \Theta_{n-} \exp\left\{2\pi\ic\left(\am\big(\infty^{(0)}\big) - \am(\z_n)\right)\right\},
\]
which gives \eqref{Thetan-jump} by \eqref{jip}.
\end{proof}

\subsection{Proof of Theorem~\ref{thm:NS} and Proposition~\ref{prop:Ne}}

Proposition~\ref{prop:Ne} has been proven as a part of Lemma~\ref{lem:JIP}. To prove Theorem~\ref{thm:NS}, define
\begin{equation}
\label{Psin}
\Psi_n := \Phi^nA_{n\tau+m_n}S_h\Theta_n
\end{equation}
using \eqref{Phi}, \eqref{Asigma}, \eqref{Sh}, and \eqref{Thetan}. The meromorphy properties follow straight from Lemmas~\ref{lem:Phi}, \ref{lem:Asigma}, \ref{lem:Sh}, and \ref{lem:Thetan} with \eqref{Psin-jump} specifically being the combination of \eqref{Phi-jumps}, \eqref{Asigma-jumps}, \eqref{Sh-jumps} and \eqref{Thetan-jump}. The behavior \eqref{psin-endp} around the ramification points of \( \RS \) is a direct consequence of \eqref{Sh-endp}. Now, if \( \Psi \) is a function as described in the statement of the theorem, then \( \Psi/\Psi_n \) is a rational function on \( \RS \) with a single possible pole at \( \z_n \). As \( \RS \) has genus \( 1 \), there are no rational functions on  \( \RS \) with a single pole. Hence, the ratio \( \Psi/\Psi_n \) must be a constant.

\subsection{Proof of Theorem~\ref{thm:NS2}}

It follows from Lemma~\ref{lem:Phi} that the described function \( g(z) \) is given by
\[
g(z) := \log|\Phi(\z)| = \log|z| - \log|\Phi(\z^*)|, \quad \z \in D_{\mathcal Q},
\]
where the second equality is a direct consequence of \eqref{symmetry}. Hence, \eqref{calQn}, \eqref{calRn}, and \eqref{Psin} yield that
\[
\left\{
\begin{array}{ll}
|\Qn(z)| = \big| \big(A_{n\tau+m_n}S_h\Theta_n\big)(\z) \big| e^{ng(z)}, & \z \in D_{\mathcal Q}, \medskip \\
|\Rn(z)| = \big| \big(A_{n\tau+m_n}S_h\Theta_n\big)(\z) \big| e^{-ng(z)}, & \z \in D_{\mathcal R}.
\end{array}
\right.
\]
Notice that the range of Abel's map \( \am(\z) \) is bounded. Recall also that \( \im(\mathbf B)>0 \). Therefore, it follows from \eqref{jip} that the sequence of numbers \( \{ n\tau+m_n \} \) is bounded. Hence, for any \( \delta>0 \), there exists a constant \( C_\delta>1 \) such that
\begin{equation}
\label{AS-bound}
C_\delta^{-1}\leq\big| \big(A_{n\tau+m_n}S_h\big)(\z) \big| \leq C_\delta
\end{equation}
for \( \z \) outside of circular neighborhoods of ``radius'' \( \delta \) around each ramification point. Further, compactness of \( \RS \) and continuity of the Abel's map imply that the family of functions \( \{ \Theta(\cdot;\p ) \} \), where
\[
\Theta(\z;\p) := \frac{\theta\left(\am(\z) - \am(\p) - \frac{1+\mathbf B}2\right)}{\theta\left(\am(\z) - \am\big(\infty^{(0)}\big) - \frac{1+\mathbf B}2\right)},
\]
is also compact and therefore necessarily has uniformly bounded above moduli for \( \z \) away from \( \infty^{(0)} \). Analogously, one can see that the family \( \{\Theta(\cdot;\p)/\Phi\} \) has uniformly bounded moduli away from \( 0^{(1)} \). The last two observations finish the proof of the upper  bounds in \eqref{calQn-bound} and \eqref{calRn-bound}. Clearly, the lower bounds amount to estimating the moduli of \( \Theta(\z;\p) \) and \( \Theta(\z;\p)/\Phi \) from below outside of \( D_\delta(\p) \). The existence of a such a bound for each function is obvious, the fact the infimum of these bounds is positive follows again from compactness.

Further, observe that  the other zero of \( \Upsilon_n \) is \( \z_n^\star \) by Lemma~\ref{lem:JIP}. Therefore, we get from \eqref{jip} that
\[
\am(\z_n^\star) = 2\am\big(\infty^{(0)}\big) - \am\big(0^{(1)}\big) -\frac1{2\pi\ic}\oint_{\boldsymbol\Delta}\lambda_h\mathcal H +n (\omega+\mathbf B\tau) + j_n^\star + \mathbf Bm_n^\star 
\]
for some \( j_n^\star,m_n^\star \in\Z \). Hence, the function \( \Psi_n^\star \) can be equivalently defined as
\begin{equation}
\label{Psin-star}
\Psi_n^\star = \Phi^n A_{n\tau+m_n^\star}S_h\Theta\big(\cdot;0^{(1)}\big)\Theta\big(\cdot;\z_n^\star\big).
\end{equation}
From this representation we can obtain bounds \eqref{calQn-bound} and \eqref{calRn-bound} exactly as before.  Lastly, notice that the ratio \( \Psi_{n-1}^\star/\Psi_n \) is equal to \( \mathcal Q_{n-1}^\star/\Qn \) in \( D_{\mathcal Q} \) and  to \( \mathcal R_{n-1}^\star/\Rn \) in \( D_{\mathcal R} \).  Hence, we just need to estimate \( |\Psi_{n-1}^\star/\Psi_n| \) on \( \RS \). It clearly follows from \eqref{Psin} and \eqref{Psin-star} that
\[
|\Psi_{n-1}^\star/\Psi_n| = |A_{m_{n-1}^\star-m_n-\tau}| \cdot | \Theta\big(\cdot;0^{(1)}\big)/\Phi | \cdot | \Theta\big(\cdot;\z_{n-1}^\star\big)/\Theta_n|.
\]
Similarly to \eqref{AS-bound}, we can argue that the first term in the above product is uniformly bounded above with \( n \) on the whole surface \( \RS \). The middle term is a single function with a simple pole at \( \infty^{(0)} \). Finally, the last ratio has a single pole at \( \z_n \) and therefore is uniformly bounded above in \( \RS\setminus D_\delta(\z_n) \) for any \( \delta> 0 \) by the previous compactness argument.

\section{Proof of Theorem~\ref{thm:main1} when \( \rho\in\mathcal W_2 \)}

To analyze the asymptotic behavior of the polynomials $Q_n$ and linearized error functions $R_n$, we use the matrix Riemann-Hilbert approach pioneered by Fokas, Its, and Kitaev \cite{FIK91,FIK92} and the non-linear steepest descent method developed by Deift and Zhou \cite{DZ93}. In what follows, it will be convenient to set
\[
\sigma_3 := \left(\begin{matrix} 1 & 0 \\ 0 & -1 \end{matrix}\right) \qandq \boldsymbol I :=\left(\begin{matrix} 1 & 0 \\ 0 & 1 \end{matrix}\right).
\]

\subsection{Orthogonality}
\label{ssec:ortho}

We shall also need the two-point Pad\'e approximant to \( f_\rho \) of type \( (n,n-1)\), which we denote by \( P_{n-1}^\star/Q_{n-1}^\star \). Set
\begin{equation}
\label{Rnstar}
R_{n-1}^\star(z) := z^{-n}(Q_{n-1}^\star f_\rho-P_{n-1}^\star)(z), \quad z\in\overline\C\setminus F.
\end{equation}
According to \eqref{Pade} and \eqref{Rn}, the functions \( R_n,R_{n-1}^\star \) are holomorphic around the origin and it holds that
\begin{equation}
\label{Rn1}
R_n(z) = \mathcal{O}(z^{-n-1}), \quad R_{n-1}^\star(z) = \mathcal{O}(z^{-n}) \quad \text{as} \quad z\to\infty.
\end{equation}
Let \( \Omega \) be a bounded annular domain containing $F$ and not containing $0$, whose boundary consists of two smooth Jordan curves. Assuming $\partial \Omega$ to be positively oriented, we get that
\begin{equation}
\label{ortho}
0 = \int_{\partial\Omega} R_n(z)z^k\dd z = \int_{\partial \Omega} Q_n(z)f_\rho(z)z^{k-n}\dd z = -\int_FQ_n(s)s^{k-n}\rho(s)\dd s
\end{equation}
for any $k\in\{0,\ldots,n-1\}$, where the first equality follows from \eqref{Rn1} and the Cauchy theorem applied outside of $\Omega$, the second is obtained by applying Cauchy theorem inside of $\Omega$, and the last is a consequence of \eqref{f}, Fubini-Tonelli's theorem, and the Cauchy integral formula. Analogously to \eqref{ortho} we get that
\begin{equation}
\label{ortho*}
0 = \int_FQ_{n-1}^\star(s)s^{k-n}\rho(s)\dd s, \quad k\in\{0,\ldots,n-2\}.
\end{equation} 
Moreover, similar computation also yields that
\[
R_{n-1}^\star(z) = -z^{-n}\int_F Q_{n-1}^\star(s)s^{-1}\rho(s)\frac{\dd s}{2\pi\mathrm i} + \mathcal{O}\big(z^{-n-1}\big) =:  \frac1{a_nz^n} + \mathcal{O}\big(z^{-n-1}\big).
\]
Hence, $a_n$ is infinite if and only if $Q_{n-1}^\star$ satisfies \eqref{ortho*} with $k=n-1$ as well. However, if the latter is true, then $Q_{n-1}^\star$ satisfies \eqref{ortho}. Conversely, if there exists a polynomial of degree at most $n-1$ satisfying \eqref{ortho}, it automatically satisfies \eqref{ortho*} and therefore the coefficient next to $z^{-n}$ in the expansion of $R_{n-1}^\star$ at infinity must be zero. Altogether, $a_n$ is finite if and only if $\deg(Q_n)=n$, where $Q_n$ is the smallest degree polynomials satisfying \eqref{ortho}.

\subsection{Initial RH Problem}
\label{ssec:irh}

Under the assumption $\deg(Q_n)=n$, define
\begin{equation}
\label{Y}
{\boldsymbol Y}:=
\left(
\begin{matrix}
Q_n & R_n \smallskip \\
a_nQ_{n-1}^\star & a_nR_{n-1}^\star
\end{matrix}
\right).
\end{equation}
Then this matrix solves the following Riemann-Hilbert problem (\rhy): find a $2\times2$ matrix-valued function $\boldsymbol Y$ such that
\begin{itemize}
\label{rhy}
\item[(a)] ${\boldsymbol Y}$ is analytic in $\C\setminus F$ and $\displaystyle \lim_{z\to\infty} {\boldsymbol Y}(z)z^{-n\sigma_3} = \boldsymbol I$;
\item[(b)] $\boldsymbol Y$ has continuous traces on $F^\circ$ that satisfy \( \boldsymbol Y_+(s) = \boldsymbol Y_- (s)\left( \begin{matrix} 1 & \rho(s)/s^n \\ 0 & 1 \end{matrix} \right) \);
\item[(c)] it holds that \( \boldsymbol Y(z) = \mathcal O\left(\begin{matrix} 1 & 1 \\ 1 & 1 \end{matrix}\right) \) as \( z\to b \) and \( \boldsymbol Y(z) = \mathcal O\left(\begin{matrix} 1 & |z-e|^{-1/2} \\ 1 & |z-e|^{-1/2} \end{matrix}\right) \) as \( z\to e\in\big\{a,a^{-1}\big\} \), where \( \mathcal O(\cdot) \) is understood entrywise.
\end{itemize}
Indeed, it is straightforward that $\boldsymbol Y$ fulfills \hyperref[rhy]{\rhy}(a) given that $\deg(Q_n)=n$, which also implies that $a_n$ is finite. Let $Q$ be either $Q_n$ or $Q_{n-1}^\star$ and $R$ be either $R_n$ or $R_{n-1}^\star$. Then we deduce from \eqref{Rn}, \eqref{Rnstar}, and the Sokhotski-Plemelj formulae \cite[Section~I.4.2]{Gakhov} that
\[
\big(R_+ - R_-\big)(x) = (Q\rho)(x)/x^n, \quad x\in F^\circ, 
\]
and therefore $\boldsymbol Y$ fulfills \hyperref[rhy]{\rhy}(b). Finally, it follows from \eqref{h-real} that
\begin{equation}
\label{sum-cond-real}
\rho_{|F_a}(b) - \rho_{|F_{a^{-1}}}(b) + \rho_{|F_{-b}}(b+) - \rho_{|F_{-b}}(b-) = \frac{h(b)}{w_+(b)} - \frac{h(b)}{w_-(b)} + \frac{h(b)}{w_-(b)} - \frac{h(b)}{w_+(b)} = 0,
\end{equation}
where the limits \( \rho_{|F_{-b}}(b\pm) \) are evaluated in accordance with the orientation of \( F_{-b} \) (also keep in mind that the segment \( [a^{-1},a ] \) is always oriented from \( a^{-1} \)  to \( a \)). Thus, \hyperref[rhy]{\rhy}(c) follows from the known behavior of Cauchy integrals near points of discontinuity of the weight \cite[Sections~I.8.1--4]{Gakhov}, where the fact that the second column does not have a logarithmic singularity around \( b \) is a direct consequence of \eqref{sum-cond-real}. To show that a solution of \hyperref[rhy]{\rhy}, if exists, must be of the form \eqref{Y} is by now a standard exercise, see for instance, \cite[Lemma~2.3]{KMcLVAV04} or \cite[Lemma~1]{ApY15}. Thus, we proved the following lemma.

\begin{lemma}
\label{lem:rhy}
If a solution of \hyperref[rhy]{\rhy} exists, then it is unique and is given by \eqref{Y} where $\deg(Q_n)=n$. Conversely, if $\deg(Q_n)=n$, then \eqref{Y} solves \hyperref[rhy]{\rhy}.
\end{lemma}

\subsection{Opening of Lenses}
\label{ssec:ol}

Let \( \Gamma_0 \) and \( \Gamma_\infty \) be two positively oriented Jordan curves that lie in \( D_0 \) and \( D_\infty \), respectively. Assume further that these curves are close enough to \( F \) so that \( h(z) \) is holomorphic and non-vanishing on the annular domain bounded by them. Denote by \( \Omega_0 \) and \( \Omega_\infty \) the intersection of this annular domain with \( D_0 \) and \( D_\infty \), respectively. Define
\begin{equation}
\label{X}
\boldsymbol X(z) := \boldsymbol Y(z)\left\{
\begin{array}{rl}
\left(\begin{matrix} 1 & 0 \\ -z^n/\rho(z) & 1  \end{matrix}\right), & z\in\Omega_0, \medskip \\
\left(\begin{matrix} 1 & 0 \\ z^n/\rho(z) & 1  \end{matrix}\right), & z\in\Omega_\infty, \medskip \\
\boldsymbol I, & z\in \overline\C\setminus\big(\overline\Omega_0\cup\overline\Omega_\infty\big),
\end{array}
\right.
\end{equation}
where we set \( \rho(z) := h(z)/w(z) \) in \( \Omega_0 \cup \Omega_\infty \) for \( w(z) \) given by \eqref{wa}. Then the matrix \( \boldsymbol X \) solves the following Riemann-Hilbert problem (\rhx):
\begin{itemize}
\label{rhx}
\item[(a)] $\boldsymbol X$ is analytic in $\C\setminus (\Gamma_0\cup F\cup\Gamma_\infty)$ and $\displaystyle \lim_{z\to\infty} {\boldsymbol X}(z)z^{-n\sigma_3} = \boldsymbol I$;
\item[(b)] $\boldsymbol X$ has continuous traces on $\Gamma_0\cup F^\circ\cup\Gamma_\infty$ that satisfy
\[
\boldsymbol X_+(s) =\boldsymbol X_-(s) \left\{
\begin{array}{rl}
\left(\begin{matrix} 1 & 0 \\ s^n/\rho(s) & 1  \end{matrix}\right), & s\in\Gamma_0\cup\Gamma_\infty, \medskip \\
\left(\begin{matrix} 0 & \rho(s)/s^n \\ -s^n/\rho(s) & 0 \end{matrix}\right), & s\in F^\circ;
\end{array}
\right.
\]
\item[(c)] \( \boldsymbol X \) satisfies \hyperref[rhy]{\rhy}(c).
\end{itemize}

The following lemma trivially holds.
\begin{lemma}
\label{lem:rhx}
\hyperref[rhx]{\rhx} is solvable if and only if \hyperref[rhy]{\rhy} is solvable. When solutions of \hyperref[rhx]{\rhx} and \hyperref[rhy]{\rhy} exist, they are unique and connected by \eqref{X}.
\end{lemma}

\subsection{Model RH Problem}
\label{ssec:mrh}

Consider the following Riemann-Hilbert problem (\rhn):
\begin{itemize}
\label{rhn}
\item[(a)] $\boldsymbol N$ is analytic in $\C\setminus F$ and $\displaystyle \lim_{z\to\infty} {\boldsymbol N}(z)z^{-n\sigma_3} = \boldsymbol I$;
\item[(b)] $\boldsymbol N$ has continuous traces on $F^\circ$ that satisfy 
\[ \boldsymbol N_+(s) =\boldsymbol N_-(s) \left(\begin{matrix} 0 & \rho(s)/s^n \\ -s^n/\rho(s) & 0 \end{matrix}\right).
\]
\end{itemize}

To solve \hyperref[rhn]{\rhn}, recall the definition of \( \Qn,\Rn \) in \eqref{calQn-real} and \eqref{calRn-real}. Since \( D_+=D_-h \) on \( \T \), \( S_+S_-=D^2/h \) on \( F_a \), and \( S_+S_-=D^2h \) on \( F_{a^{-1}} \), it can be easily checked that
\begin{equation}
\label{Psin-jumps}
\mathcal Q_{n\pm}(s) = \big[s^n/h(s)\big] \mathcal R_{n\mp}(s)\left\{
\begin{array}{rl}
1, & s\in F_a^\circ, \medskip \\
-1, & s\in F_{a^{-1}}^\circ, \medskip \\
\mp1, & s\in F_{-b}^\circ,
\end{array}
\right.
\end{equation}
where we also used \eqref{equilibrium}. Further, set
\begin{equation}
\label{phi}
\phi(z) := \frac{z-b-w(z)}{\sqrt a-\sqrt{a^{-1}}},
\end{equation}
where the convention concerning the roots of negative numbers is the same as in \eqref{varphi}. Similarly to \( \varphi(z) \), one can see that \( \phi(z) \) is holomorphic off \( [a^{-1},a] \), has a simple zero at the origin, and satisfies \( \phi_-(s)\phi_+(s)=s \) for \( s\in [a^{-1},a] \). Further, with \eqref{calQn-real} and \eqref{calRn-real} at hand, let us put
\begin{equation}
\label{calQn-star-real}
\mathcal Q_{n-1}^\star(z) := \mathcal Q_{n-1}(z)\left\{
\begin{array}{rl}
z/\phi(z), & z\in D_0, \medskip \\ 
\phi(z), & z\in D_\infty,
\end{array}
\right.
\end{equation}
which is a holomorphic and non-vanishing function in \( \C\setminus F \) with a pole of order \( n-1 \) at infinity and
\begin{equation}
\label{calRn-star-real}
\mathcal R_{n-1}^\star(z) := \mathcal R_{n-1}(z)\left\{
\begin{array}{rl}
\phi(z)/z, & z\in D_0, \medskip \\ 
1/\phi(z), & z\in D_\infty,
\end{array}
\right.
\end{equation}
which is a holomorphic and non-vanishing function in \( \C\setminus F \) with a zero of multiplicity \( n-1 \) at infinity. Clearly, \( \mathcal Q_{n-1}^\star \) and \( \mathcal R_{n-1}^\star \) also satisfy \eqref{Psin-jumps}. Then it can be readily checked that
\begin{equation}
\label{N}
\boldsymbol N := \boldsymbol{CM}, \quad \boldsymbol C:= \left(\begin{matrix} \gamma_n & 0 \medskip \\ 0 & \gamma_{n-1}^\star  \end{matrix}\right), \quad \boldsymbol M := \left( \begin{matrix} \Qn & \Rn/w \medskip \\ \mathcal Q_{n-1}^\star & \mathcal R_{n-1}^\star/w \end{matrix} \right),
\end{equation}
solves  \hyperref[rhn]{\rhn}, where \( \gamma_{n-1}^\star \) is a constant such that \( \lim_{z\to\infty}\gamma_{n-1}^\star z^{n-1}\mathcal R_{n-1}^\star(z) = 1 \). Observe that \( \boldsymbol M \) satisfies \hyperref[rhy]{\rhy}(c) and \( \det(\boldsymbol N)\equiv1 \) because \( \det(\boldsymbol N)(z) \) is a holomorphic function outside \( \big\{a,a^{-1}\big\} \), where it has at most square root singularities, and that has value \( 1 \) at infinity. It also holds that
\begin{equation}
\label{detM-real}
\det(\boldsymbol M) = \big(\gamma_n\gamma_{n-1}^\star\big)^{-1} = -\frac4{2b+a+a^{-1}}.
\end{equation}

\subsection{RH Problem with Small Jumps}
\label{ssec:sj}

Consider the following Riemann-Hilbert problem (\rhz):
\begin{itemize}
\label{rhz}
\item[(a)] $\boldsymbol{Z}$ is a holomorphic matrix function in $\overline\C\setminus(\Gamma_0\cup \Gamma_\infty)$ and $\boldsymbol{Z}(\infty)=\boldsymbol{I}$;
\item[(b)] $\boldsymbol{Z}$ has continuous traces on $\Gamma_0\cup \Gamma_\infty$ that satisfy
\[
\boldsymbol Z_+(s)  = \boldsymbol Z_- (s)\boldsymbol M(s) \left(\begin{matrix} 1 & 0 \\ s^n/\rho(s) & 1 \end{matrix}\right) \boldsymbol M^{-1}(s).
\]
\end{itemize}
Then the following lemma takes place.
\begin{lemma}
\label{lem:rhz1}
For \( n \) large enough, a solution of \hyperref[rhz]{\rhz} exists and satisfies
\begin{equation}
\label{Z}
\boldsymbol{Z}=\boldsymbol{I}+ \mathcal O\big( c^n\big)
\end{equation}
for some constant $c<1$ independent of $\Gamma_0,\Gamma_\infty$, where $\mathcal O(\cdot)$ holds uniformly in $\overline\C$.
\end{lemma}
\begin{proof}
It follows from an explicit computation and \eqref{detM-real} that the jump matrix for $\boldsymbol{Z}$ is equal to
\begin{equation}
\label{Zjump}
\boldsymbol I + \gamma_n\gamma_{n-1}^\star\frac{s^n}{(hw)(s)} \left(\begin{matrix} (\Rn\mathcal R_{n-1}^\star)(s) & -\Rn^2(s) \medskip \\ \mathcal R_{n-1}^{\star2}(s) &  -(\Rn\mathcal R_{n-1}^\star)(s) \end{matrix}\right) = \boldsymbol I  + \mathcal O\big( c^n(\Gamma_0,\Gamma_\infty) \big),
\end{equation}
where \( c(\Gamma_0,\Gamma_\infty)\in (0,1) \) and the last equality is a consequence of \eqref{equilibrium} and the maximum modulus principle for holomorphic  functions.The conclusion of the lemma now follows from the same argument as in \cite[Corollary~7.108]{Deift}.
\end{proof}

\subsection{Asymptotics}

Let $\boldsymbol Z$ be a solution of \hyperref[rhz]{\rhz} granted by Lemma~\ref{lem:rhz1} and $\boldsymbol C,\boldsymbol M$ be defined by \eqref{N}. Then it can be easily checked that \( \boldsymbol X := \boldsymbol{CZM} \) solves \hyperref[rhx]{\rhx} and therefore the solution of \hyperref[rhy]{\rhy} is obtained from \eqref{X}.

Given any closed set $B\subset\overline\C\setminus F$, choose $\Omega_0,\Omega_\infty$ so that $B\subset \overline\C \setminus (\overline\Omega_0 \cup \overline\Omega_\infty)$. Then \( \boldsymbol Y = \boldsymbol{CZM} \) on \( B \). Hence, if the first row of \( \boldsymbol Z \) is denoted by \( \left(\begin{matrix} 1+\upsilon_{n1} & \upsilon_{n2} \end{matrix}\right) \), we have that
\[
\left\{
\begin{array}{rll}
Q_n &=& \gamma_n\left[\big(1+\upsilon_{n1}\big)\Qn + \upsilon_{n2}\mathcal Q_{n-1}^\star\right], \medskip \\
wR_n & = & \gamma_n\left[\big(1+\upsilon_{n1}\big)\Rn + \upsilon_{n2}\mathcal R_{n-1}^\star\right],
\end{array}
\right.
\]
by  \eqref{Y} and \eqref{N}. Equations \eqref{asymptotics1} now follow from \eqref{calQn-real} and \eqref{calQn-star-real} together with \eqref{calRn-real} and \eqref{calRn-star-real} since we know from Lemma~\ref{lem:rhz1} that $|\upsilon_{nk}|\leq c^n$ uniformly in $\overline\C$ ($\upsilon_{nk}(\infty)=0$ as $\boldsymbol{Z}(\infty)=\boldsymbol{I}$).

\section{Proof of Theorem~\ref{thm:main2} when \( \rho\in\mathcal W_2 \)}

It is straightforward to check that everything written in Sections~\ref{ssec:ortho}-\ref{ssec:ol} remains valid except for \hyperref[rhy]{\rhy}(c)  which now simply reads 
\[
\boldsymbol Y(z) = \mathcal O\left(\begin{matrix} 1 & |z-e|^{-1/2} \\ 1 & |z-e|^{-1/2} \end{matrix}\right) \quad \text{as} \quad z\to e\in E=\big\{a,b,a^{-1},b^{-1}\big\}.
\]
Furthermore, the formulation of \hyperref[rhn]{\rhn} remains the same as well. To solve it, observe that \eqref{Psin-jumps} still holds (one needs to replace \( F_{-b}^\circ \) with \( F_{-1}^\circ\cup F_1^\circ \)), where the functions \( \Qn, \Rn \) are now defined by \eqref{calQn} and \eqref{calRn}. Indeed, for \( s\in F_a^\circ \), it holds that
\[
\mathcal Q_{n\pm}(s) = \Psi_{n+}(\s) = h^{-1}(s)\Psi_{n-}(\s) = \big[s^n/h(s)\big]\mathcal R_{n\mp}(s)
\]
as claimed. The proof of \eqref{Psin-jumps} on the rest of the arcs is absolutely analogous (one just needs to pay attention to the chosen orientations of the cycles \( \boldsymbol\alpha,\boldsymbol\beta,\boldsymbol\gamma,\boldsymbol\delta \)).  Moreover, the functions \( \mathcal Q_{n-1}^\star \) and \( \mathcal R_{n-1}^\star \), defined just before Theorem~\ref{thm:main2}, satisfy \eqref{Psin-jumps} as well, which can be shown in a similar fashion since \( \Psi_{n-1}^\star \) obviously satisfies~\eqref{Psin-jump}. Hence, it is easy to check using \eqref{Psin-jumps} that \hyperref[rhn]{\rhn} is solved  for each \( n \) such that \( \z_n\neq\infty^{(1)} \) by \eqref{N}, where the constants \( \gamma_n \) and \( \gamma_{n-1}^\star \) are again defined by
\[
\lim_{z\to\infty}\gamma_n\Qn(z)z^{-n} = 1 \qandq \lim_{z\to\infty} \gamma_{n-1}^\star\mathcal R_{n-1}^\star(z)z^{n-2} = 1.
\]
These constants are well defined by the very definition of \( \N_\varepsilon \) and Lemma~\ref{lem:JIP}. Notice again that \( \boldsymbol M \) has the same behavior near \( E \) as \( \boldsymbol Y\). Moreover, \( \det(\boldsymbol N)\equiv1 \) due to the same reasons as before, and therefore \( \det(\boldsymbol M) = (\gamma_n\gamma_{n-1}^\star)^{-1} \). Given the solution of \hyperref[rhn]{\rhn},  we again can formulate \hyperref[rhz]{\rhz}. Obviously, the jump of \( \boldsymbol Z \) is equal to the left-hand side of \eqref{Zjump}. Since 
\[
\gamma_n\gamma_{n-1}^\star = \lim_{z\to\infty}\frac{z^2}{\Qn(z)\mathcal R_{n-1}^\star(z)} = \lim_{\z\to\infty^{(1)}} \frac{z^3A_{m_{n-1}^\star-m_n-\tau}(\z^*)}{\Phi(\z)\Theta_n(\z)\Theta\big(\z^*;0^{(1)}\big)\Theta\big(\z^*;\z_{n-1}^\star\big)},
\]
it follows from the very definition of \( \N_\varepsilon \), Lemma~\ref{lem:JIP}, and the compactness argument from the proof of Theorem~\ref{thm:NS2} that the sequence \( \{|\gamma_n\gamma_{n-1}^\star|\}_{n\in\N_\varepsilon} \) is bounded above (the constant does depend on \( \varepsilon \)). Therefore, the conclusion of Lemma~\ref{lem:rhz1} still holds, but only for all \( n\in\N_\varepsilon \) large enough and with constant \( c=c_\varepsilon \), where we need to use \eqref{calRn-bound} and \eqref{fun-g} coupled with the maximum principle for harmonic functions to show the equality in \eqref{Zjump}. Finally, the proof of \eqref{asymptotics2} is now absolutely the same as in the case of Theorem~\ref{thm:main1}.

\section{Proof of Theorem~\ref{thm:main1} when \( \rho\in\mathcal W_1 \)}

\subsection{Initial RH Problem}

The material of Section~\ref{ssec:ortho} remains valid. The only change in Section~\ref{ssec:irh} needs to be made in \hyperref[rhy]{\rhy}(c) that is replaced by
\begin{equation}
\label{Yc}
\boldsymbol Y(z) = \left\{ \begin{array}{rl}
\mathcal O \left(\begin{matrix} 1 & 1 \\ 1 & 1 \end{matrix}\right) & \text{as} \quad z\to\big\{b,b^{-1}\big\}, \medskip \\
\mathcal O \left(\begin{matrix} 1 & \psi_\alpha(z-a) \\ 1 & \psi_\alpha(z-a) \end{matrix}\right) & \text{as} \quad z\to a, \medskip \\
\mathcal O \left(\begin{matrix} 1 & \psi_\beta(z-a^{-1}) \\ 1 & \psi_\beta(z-a^{-1}) \end{matrix}\right) & \text{as} \quad z\to a^{-1},
\end{array}
\right.
\end{equation}
 where
\[
\psi_\alpha(z) :=
\left\{
\begin{array}{ll}
|z|^\alpha, & \mbox{if} \quad \alpha<0, \smallskip \\
\log|z|, & \mbox{if} \quad \alpha=0,\smallskip \\
1, & \mbox{if} \quad \alpha>0.
\end{array}
\right.
\]

\subsection{Opening of Lenses}

Here, we choose \( \Gamma_0,\Gamma_\infty \) as in Section~\ref{ssec:ol} with the exception of requiring \( \Gamma_0 \) to touch \( F \) at \( a \) and \( \Gamma_\infty \) to touch \( F \) at \( a^{-1} \). We define \( \Omega_0,\Omega_\infty \) again as in Section~\ref{ssec:ol}, however, now they are no longer annular domains. Further, we still define \( \boldsymbol X \) by \eqref{X} with \( \rho(s) \) extended to \( \Omega_0\cup\Omega_\infty \) by \( h(z)/w(z) \) (we assume that the branch cuts of \( (z-a)^{\alpha+1/2} \) and \( (z-a^{-1})^{\beta+1/2} \) in \eqref{W1} lie outside of some neighborhoods of \( a \) and \( a^{-1} \) intersected with \(\overline\Omega_0\cup\overline\Omega_\infty\)). The Riemann-Hilbert problem \hyperref[rhx]{\rhx} remains the same except for \hyperref[rhx]{\rhx}(c), which needs to be modified within \( \Omega_0\cup\Omega_\infty \) as follows:
\begin{equation}
\label{Xc}
\boldsymbol X(z) = \left\{
\begin{array}{ll}
\mathcal O\left(\begin{matrix} 1 & |z-a|^\alpha \\ 1 & |z-a|^\alpha \end{matrix}\right), & \alpha<0, \medskip \\
\mathcal O\left(\begin{matrix} \log|z-a| & \log|z-a| \\ \log|z-a| & \log|z-a| \end{matrix}\right), & \alpha=0, \medskip \\
\mathcal O\left(\begin{matrix} |z-a|^{-\alpha} & 1 \\ |z-a|^{-\alpha} & 1 \end{matrix}\right), & \alpha>0,
\end{array}
\right.
\end{equation}
as \( \Omega_0\cup\Omega_\infty\ni z\to a\), and an analogous change should be made around \( a^{-1} \).  With the above changes, Lemma~\ref{lem:rhx} still holds.

\subsection{Model and Local RH Problems}
\label{ssec:lrh}

Model Riemann-Hilbert problem \hyperref[rhn]{\rhn} is formulated and solved exactly as in the case \( \rho\in\mathcal W_2 \).  Moreover,  it is still true that \( \det(\boldsymbol N)\equiv 1 \) (the singular behavior of the entries of \( \boldsymbol N \) around \( a,a^{-1} \) gets canceled when determinant is evaluated). Let now \( U_a,U_{a^{-1}} \) be open sets around \( a,a^{-1} \).  Define
\begin{equation}
\label{D1}
\boldsymbol D(z) := \left\{
\begin{array}{ll}
(z/\varphi(z))^{n\sigma_3}, & z\in D_0, \medskip \\
\varphi(z)^{n\sigma_3}, & z\in D_\infty,
\end{array}
\right.
\end{equation}
where \( \varphi \) is given by \eqref{varphi} and \( g^{\sigma_3} = \mathrm{diag}\big( g \ 1/g \big) \). We shall need to solve the following local Riemann-Hilbert problems (\rhp, \( e\in\big\{a,a^{-1}\big\} \)):
\begin{itemize}
\label{rhp}
\item[(a,b,c)] $\boldsymbol P_e$ satisfies \hyperref[rhx]{\rhx}(a,b,c) within \( U_e\);
\item[(d)] \( \boldsymbol P_e = \boldsymbol{MD}^{-1}\big(\boldsymbol I+\mathcal O(1/n) \big)\boldsymbol D \) uniformly on \( \partial U_e \).
\end{itemize}
Since the construction of \( \boldsymbol P_e \) is lengthy, we postpone it until the end of the section.

\subsection{RH Problem with Small Jumps} 

Let
\[
\Sigma = \big( \partial U_a \cup \partial U_{a^{-1}}\big) \cup \left[\big( \Gamma_0\cup \Gamma_\infty)\setminus \big( \overline U_a \cup \overline U_{a^{-1}}\big) \right].
\]
The Riemann-Hilbert problem \hyperref[rhz1]{\rhz} now needs to be formulated as follows:
\begin{itemize}
\label{rhz1}
\item[(a)] $\boldsymbol{Z}$ is a holomorphic matrix function in $\overline\C\setminus\Sigma$ and $\boldsymbol{Z}(\infty)=\boldsymbol{I}$;
\item[(b)] $\boldsymbol{Z}$ has continuous traces at the smooth points of $ \Sigma $ that satisfy
\[
\boldsymbol Z_+(s)  = \boldsymbol Z_- (s)\left\{
\begin{array}{l}
\boldsymbol M(s) \left(\begin{matrix} 1 & 0 \\ s^n/\rho(s) & 1 \end{matrix}\right) \boldsymbol M^{-1}(s), \medskip \\
\big(\boldsymbol P_e \boldsymbol M^{-1}\big)(s),
\end{array}
\right.
\]
where the first relation holds for \( s\in \big(\Gamma_0\cup \Gamma_\infty\big)\setminus \big( \overline U_a \cup \overline U_{a^{-1}}\big) \) and the second one for \(  s\in \partial U_e\setminus \big(\Gamma_0\cup F\cup \Gamma_\infty\big) \), \( e\in\big\{a,a^{-1}\big\} \).
\end{itemize}

Then the following lemma takes place.
\begin{lemma}
\label{lem:rhz1}
For all \( n \) large enough, a solution of \hyperref[rhz1]{\rhz} exists and satisfies
\(
\boldsymbol{Z}=\boldsymbol{I}+ \mathcal O\big( 1/n\big)
\)
uniformly in~$\overline\C$.
\end{lemma}
\begin{proof}
The proof of the fact that the jump of \( \boldsymbol Z \) is geometrically small on \( \big(\Gamma_0\cup \Gamma_\infty\big)\setminus \big( \overline U_a \cup \overline U_{a^{-1}}\big) \) is the same as in the case \( \rho\in\mathcal W_2 \). Furthermore, we have that
\[
\boldsymbol P_e \boldsymbol M^{-1} = \boldsymbol I + \boldsymbol{MD}^{-1}\mathcal O(1/n)\boldsymbol{DM}^{-1}
\]
on \( \partial U_e \). It follows from \eqref{calQn-real}, \eqref{calRn-real}, \eqref{calQn-star-real}, \eqref{calRn-star-real}, and \eqref{N} that
\[
\boldsymbol{MD}^{-1} = \left( \begin{matrix} 1 & 1 \medskip \\ \varphi/\phi & \phi/\varphi \end{matrix}\right)\left(\frac SD \right)^{\sigma_3}
\]
on \( \partial U_a \) and a similar formula holds on \(\partial U_{a^{-1}} \). In any case it is a fixed matrix independent of \( n \). Hence, the jump of \( \boldsymbol Z \) is of order  \( \boldsymbol I + \mathcal O(1/n) \) on \( \partial U_a \cup \partial U_{a^{-1}} \). The conclusion of the lemma now follows as in the case \( \rho\in\mathcal W_2 \).
\end{proof}

\subsection{Asymptotics} 

Formulae \eqref{asymptotics1} follow now exactly as in the case \( \rho\in\mathcal W_2 \).

\subsection{Solution of \hyperref[rhp]{\rhp}, \( e\in\big\{a,a^{-1}\big\} \)}
\label{ssec:hard}

We shall construct the matrix \( \boldsymbol P_a \) only as the construction of \(  \boldsymbol P_{a^{-1}} \) is completely similar.

\subsubsection{Model Problem}

Below, we always assume that the real line as well as its subintervals are oriented from left to right. Further, we set
\begin{equation}
\label{rays}
I_\pm:=\big\{z:\arg(z)=\pm2\pi/3\big\},
\end{equation}
where the rays $I_\pm$ are oriented towards the origin. Given $\alpha>-1$, let $\boldsymbol\Psi_\alpha$ be a matrix-valued function such that 
\begin{itemize}
\label{rhpsiA}
\item[(a)] $\boldsymbol\Psi_\alpha$ is holomorphic in $\C\setminus\big(I_+\cup I_-\cup(-\infty,0]\big)$;
\item[(b)] $\boldsymbol\Psi_\alpha$ has continuous traces on $I_+\cup I_-\cup(-\infty,0)$ that satisfy
\[
\boldsymbol\Psi_{\alpha+} = \boldsymbol\Psi_{\alpha-}
\left\{
\begin{array}{rll}
\left(\begin{matrix} 0 & 1 \\ -1 & 0 \end{matrix}\right) & \text{on} & (-\infty,0), \medskip \\
\left(\begin{matrix} 1 & 0 \\ e^{\pm\pi\mathrm{i}\alpha} & 1 \end{matrix}\right) & \text{on} & I_\pm;
\end{array}
\right.
\]
\item[(c)] as $\zeta\to0$  it holds that
\[
\boldsymbol\Psi_\alpha(\zeta) = \mathcal{O}\left( \begin{matrix} |\zeta|^{\alpha/2} & |\zeta|^{\alpha/2} \\ |\zeta|^{\alpha/2} & |\zeta|^{\alpha/2} \end{matrix} \right) \quad \text{and} \quad \boldsymbol\Psi_\alpha(\zeta) = \mathcal{O}\left( \begin{matrix} \log|\zeta| & \log|\zeta| \\ \log|\zeta| & \log|\zeta| \end{matrix} \right) 
\]
when $\alpha<0$ and $\alpha=0$, respectively, and
\[
\boldsymbol\Psi_\alpha(\zeta) = \mathcal{O}\left( \begin{matrix} |\zeta|^{\alpha/2} & |\zeta|^{-\alpha/2} \\ |\zeta|^{\alpha/2} & |\zeta|^{-\alpha/2} \end{matrix} \right) \quad \text{and} \quad \boldsymbol\Psi_\alpha(\zeta) = \mathcal{O}\left( \begin{matrix} |\zeta|^{-\alpha/2} & |\zeta|^{-\alpha/2} \\ |\zeta|^{-\alpha/2} & |\zeta|^{-\alpha/2} \end{matrix} \right)
\]
when $\alpha>0$, for $|\arg(\zeta)|<2\pi/3$ and $2\pi/3<|\arg(\zeta)|<\pi$, respectively;
\item[(d)] it holds uniformly in $\C\setminus\big(I_+\cup I_-\cup(-\infty,0]\big)$ that
\[
\boldsymbol\Psi_\alpha(\zeta) = \boldsymbol S(\zeta)\left(\boldsymbol I+\mathcal{O}\left(\zeta^{-1/2}\right)\right)\exp\left\{2\zeta^{1/2}\sigma_3\right\}
\]
where \( \displaystyle \boldsymbol S(\zeta) := \frac{\zeta^{-\sigma_3/4}}{\sqrt2}\left(\begin{matrix} 1 & \ic \\ \ic & 1 \end{matrix}\right) \) and we take the principal branch of \( \zeta^{1/4} \).
\end{itemize}
Explicit construction of this matrix can be found in \cite{KMcLVAV04} (it uses modified Bessel and Hankel functions). Observe that
\begin{equation}
\label{S}
\boldsymbol S_+(\zeta) = \boldsymbol S_-(\zeta)\left(\begin{matrix} 0 & 1 \\ \ -1 & 0 \end{matrix}\right),
\end{equation}
since the principal branch of \( \zeta^{1/4} \) satisfies $\zeta_+^{1/4}=\ic\zeta^{1/4}_-$.

\subsubsection{Conformal Map}

In this section we define a conformal map that will carry \( U_a \) into \( \zeta \)-plane. Set
\begin{equation}
\label{zeta-a-1}
\zeta_a(z) := \left(\frac14 \log\big(z/\varphi^2(z)\big)\right)^2, \quad z\in U_a,
\end{equation}
where the function \( \varphi \) is given by \eqref{varphi}. It follows from \eqref{equilibrium} that \( \zeta_a \) is holomorphic across \( F_a \). It also follows from the explicit representation of \( \varphi \) that \( \zeta_a \) vanishes at \( a \). Moreover, since 
\[
\frac z{\varphi^2(z)} = 1 - \frac{2w(z)}{z+b+w(z)},
\]
the zero of \( \zeta_a \) at \( a \) is necessarily simple. Notice also that \( |\varphi_+|=|\varphi_-| \) on \( [a^{-1},a] \) and therefore \( |s/\varphi_\pm^2(s)|  \equiv 1 \) there according to \eqref{equilibrium}. Hence, \( \zeta_a \) maps \( F_a \) into the negative reals. It is also simple to check that the rest of the reals in \( U_a \) are mapped into the positive reals by \( \zeta_a \). Set
\[
U_a^\pm := U_a \cap \left\{
\begin{array}{ll}
\big\{\pm\im(z)>0\big\}, & a<0, \smallskip \\
\big\{\mp\im(z)>0\big\}, & a>0.
\end{array}
\right.
\]
It should be clear from the previous discussion that \( \zeta_a(U_a^\pm) \subset \big\{\pm\im(z)>0\big\} \). Let \( \Gamma_0^\pm : = \Gamma_0\cap U_a^\pm \). Notice that according to the chosen orientation of \( \Gamma_0 \), \( \Gamma_0^+ \) is oriented towards \( a \) and \( \Gamma_0^- \) is oriented away from \( a \). As we have had some freedom in choosing the curve \( \Gamma_0 \), we shall choose it so that \( \zeta_a(\Gamma_0^\pm) \subset I_\pm \). 

Finally, in what follows we understand under \( \zeta_a^{1/2} \) the branch given by the expression in parenthesis in \eqref{zeta-a-1} with the branch cut along \( F_a \). In particular, it holds that
\begin{equation}
\label{E-zeta-a-1}
\exp\left\{2n\zeta_a^{1/2}(z)\sigma_3\right\} = z^{-n\sigma_3/2}\boldsymbol D(z),
\end{equation}
where the matrix \( \boldsymbol D \) was defined in \eqref{D1}. Similarly, we let \( \zeta_a^{1/4} \) to be the branch that maps \( U_a \) into the sector \( |\arg(z)|<\pi/4 \). For instance, it holds that \( \zeta_{a+}^{1/4} = \ic \zeta_{a-}^{1/4} \) on \( F_a \).

\subsubsection{Matrix $\boldsymbol P_a$}

Under the conditions placed on the class \( \mathcal W_1 \), it holds that
\[
\rho(z)=\frac{h_*(z)}{w(z)}\left\{\begin{array}{ll}(a-z)^{\alpha+1/2}, & a<0, \smallskip \\ (z-a)^{\alpha+1/2}, & a>0, \end{array}\right. \quad z\in U_a\setminus[-1,1],
\]
where $h_*$ is non-vanishing and holomorphic in $U_a$, $\alpha>-1$, and the $\alpha$-roots are principal. Recall also that \( \rho \) on \( F_a \) is defined as the trace of \( \rho_{|U_a^+} \) on \( F_a \). This can be equivalently stated as
\[
\rho(z) = \pm\rho_*(z)\left\{\begin{array}{ll}(a-z)^{\alpha/2}, & a<0, \smallskip \\ (z-a)^{\alpha/2}, & a>0, \end{array}\right. \quad z\in U_a^\pm,
\]
where \( \rho_* \) is  non-vanishing and holomorphic in $U_a$. Set
\[
r_a(z) := \sqrt{\rho_*(z)}\left\{\begin{array}{ll}(z-a)^{\alpha/2}, & a<0, \smallskip \\ (a-z)^{\alpha/2}, & a>0, \end{array}\right.
\]
where the branches are again principal. Then $r_a$ is a holomorphic and non-vanishing function in $U_a\setminus F_a$ that satisfies
\[
\left\{
\begin{array}{ll}
r_{a+}(s)r_{a-}(s) = \rho(s), & s\in F_a^\circ\cap U_a, \medskip \\
r_a^2(z) = \rho(z)e^{\pi\ic\alpha}, & z\in\Gamma_0^+, \medskip \\
r_a^2(z) = -\rho(z)e^{-\pi\ic\alpha}, & z\in\Gamma_0^-.
\end{array}
\right.
\]
The above relations and \hyperref[rhpsiA]{\rhpsiA}(a,b,c) imply that
\begin{equation}
\label{Pa}
\boldsymbol P_a(z) := \boldsymbol E_a(z) \boldsymbol\Psi_\alpha\left(n^2\zeta_a(z)\right) z^{n\sigma_3/2}r_a^{-\sigma_3}(z)
\end{equation}
satisfies  \hyperref[rhp]{\rhp}(a,b,c), where $\boldsymbol E_a$ is a holomorphic matrix function (notice that the orientation of \( \zeta_a(\Gamma_0^-) \) is opposite from the orientation of \( I_- \)). It further follows from \hyperref[rhn]{\rhn}(b), \eqref{equilibrium}, and \eqref{S} that
\begin{equation}
\label{Ea}
\boldsymbol E_a(z) := \big(\boldsymbol{MD}^{-1}\big)(z)r_a^{\sigma_3}(z)\boldsymbol S^{-1}\big(n^2\zeta_a(z)\big)
\end{equation}
is holomorphic in $U_a\setminus\{a\}$. Since $|r_a(z)|\sim|z-a|^{\alpha/2}$, \( \boldsymbol S^{-1}\big(n^2\zeta_a(z)\big)\sim |z-a|^{\sigma_3/4} \), and
\[
\boldsymbol M(z) = \left(\begin{matrix} |z-a|^{-\alpha/2-1/4} & |z-a|^{\alpha/2-1/4} \medskip \\ |z-a|^{-\alpha/2-1/4} & |z-a|^{\alpha/2-1/4} \end{matrix}\right),
\]
$\boldsymbol E_a$ is in fact holomorphic in $U_a$. Finally, \hyperref[rhp]{\rhp}(d) follows now from \eqref{E-zeta-a-1} and \hyperref[rhpsiA]{\rhpsiA}(d).

\section{Proof of Theorem~\ref{thm:main2} when \( \rho\in\mathcal W_1 \)}

As usual, Sections~\ref{ssec:ortho}--\ref{ssec:irh} translate identically to the present case after \hyperref[rhy]{\rhy}(c) is replaced by~\eqref{Yc}.

\subsection{Opening of Lenses}

We choose \( \Gamma_0,\Gamma_\infty \) as in Section~\ref{ssec:ol} except for requiring \( \Gamma_0 \) to touch \( F \) at \( a \) and \( \Gamma_\infty \) to touch \( F \) at \( a^{-1} \), see Figure~\ref{F-lenses}. 
\begin{figure}[ht!]
\centering
\subfigure[\(v=b\)]{\includegraphics[scale=.5]{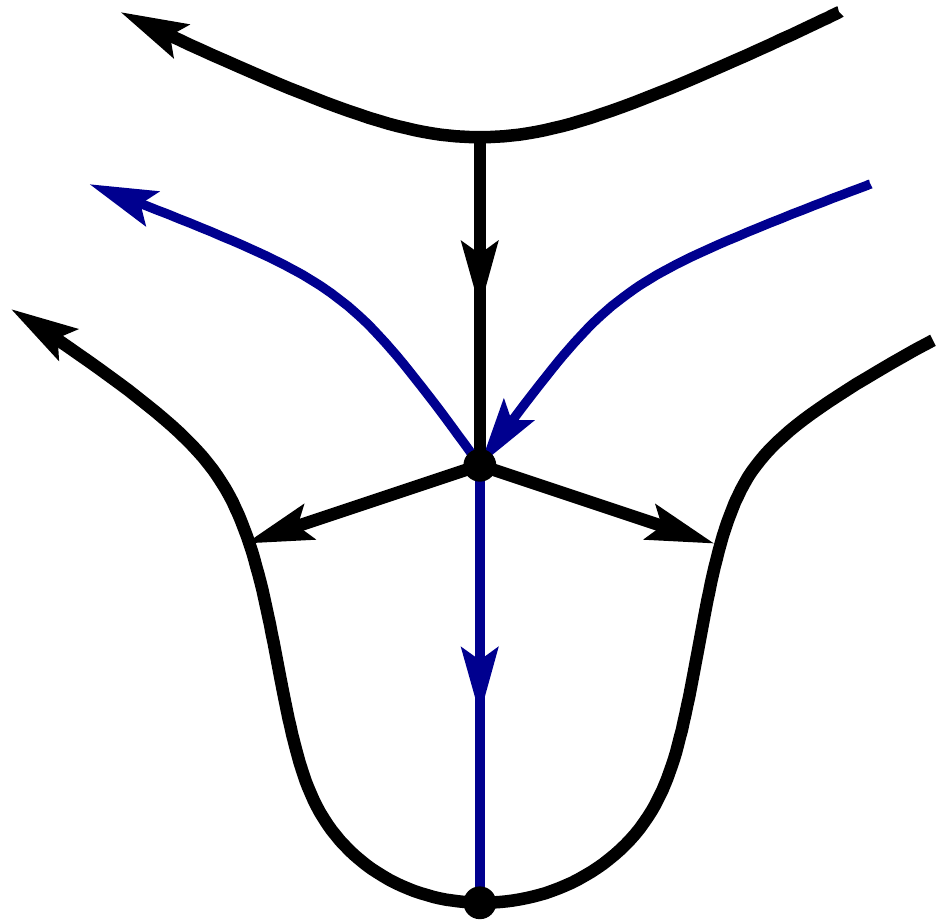}}
\begin{picture}(0,0)
\put(-68,55){$v$}
\put(-12,102){$F_1$}
\put(-143,102){$F_{-1}$}
\put(-85,37){$F_a$}
\put(-68,120){$\Gamma_\infty$}
\put(-43,10){$\Gamma_0$}
\put(-105,65){$\Gamma_{v,-1}$}
\put(-84,90){$\Gamma_v$}
\put(-65,37){$\Omega_{v,1}$}
\put(-65,102){$\Omega_{\infty,1}$}
\put(-50,75){$\Omega_{0,1}$}
\put(-110,110){\textcolor{gray}{$\oplus$}}
\put(-35,110){\textcolor{gray}{$\oplus$}}
\put(-125,90){\textcolor{gray}{$\ominus$}}
\put(-20,90){\textcolor{gray}{$\ominus$}}
\put(-83,20){\textcolor{gray}{$\oplus$}}
\put(-65,20){\textcolor{gray}{$\ominus$}}
\end{picture}
\quad\quad\quad
\subfigure[\(v=b^{-1}\)]{\includegraphics[scale=.5]{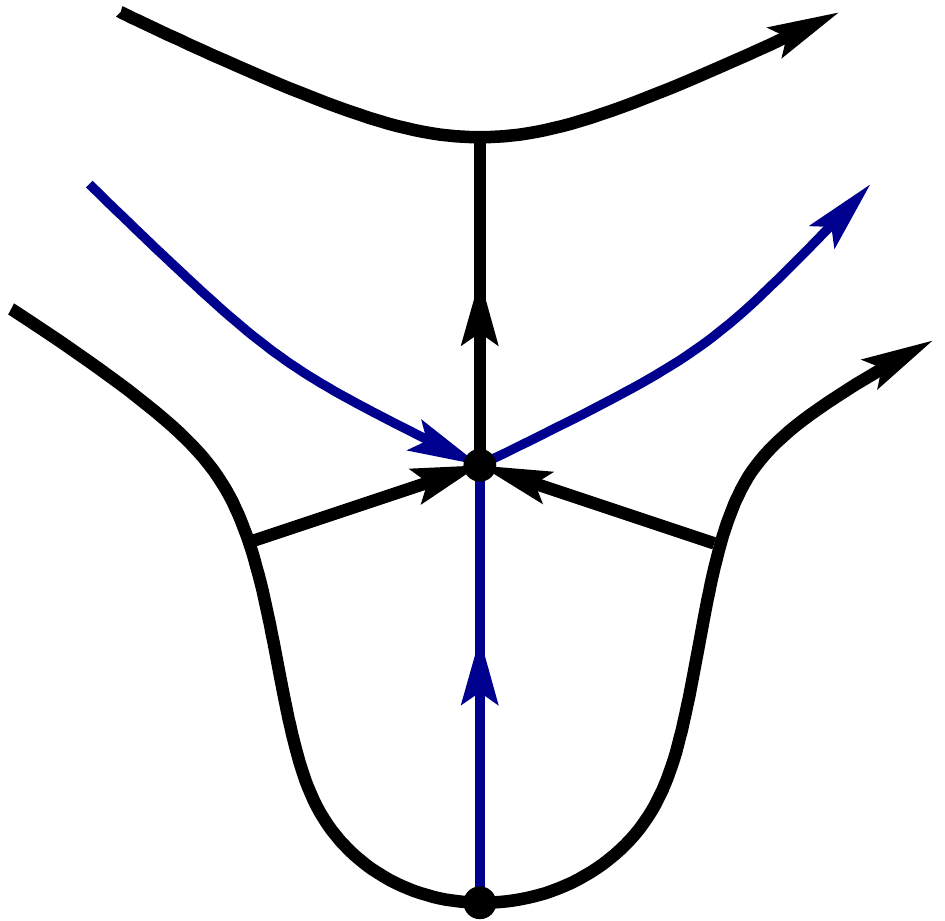}}
\begin{picture}(0,0)
\put(-68,55){$v$}
\put(-12,102){$F_1$}
\put(-143,102){$F_{-1}$}
\put(-93,37){$F_{a^{-1}}$}
\put(-68,120){$\Gamma_0$}
\put(-43,10){$\Gamma_\infty$}
\put(-105,65){$\Gamma_{v,-1}$}
\put(-84,90){$\Gamma_v$}
\put(-65,37){$\Omega_{v,1}$}
\put(-60,95){$\Omega_{0,1}$}
\put(-55,65){$\Omega_{\infty,1}$}
\put(-110,110){\textcolor{gray}{$\ominus$}}
\put(-35,110){\textcolor{gray}{$\ominus$}}
\put(-125,85){\textcolor{gray}{$\oplus$}}
\put(-20,85){\textcolor{gray}{$\oplus$}}
\put(-83,20){\textcolor{gray}{$\ominus$}}
\put(-65,20){\textcolor{gray}{$\oplus$}}
\end{picture}
\caption{\small Schematic representation of the set \( F \) (thiner arcs), arcs \( \Gamma_v\), \( \Gamma_{v,-1} \) (labeled), and \( \Gamma_{v,1} \) (not labeled, placed symmetrically across \( \Gamma_{v,-1} \)), and the domains \( \Omega_{\infty,1} \),  \( \Omega_{0,1} \), \( \Omega_{v,1} \) (labeled) and \( \Omega_{\infty,-1} \),  \( \Omega_{0,-1} \), \( \Omega_{v,-1} \) (not labeled, placed symmetrically across the labeled ones) locally around \( v\in\{b,b^{-1}\} \). The symbols \( \oplus \) and \( \ominus \) indicated whether the corresponding domain is a part of \( \Omega_+ \) or \( \Omega_- \).}
\label{F-lenses}
\end{figure}
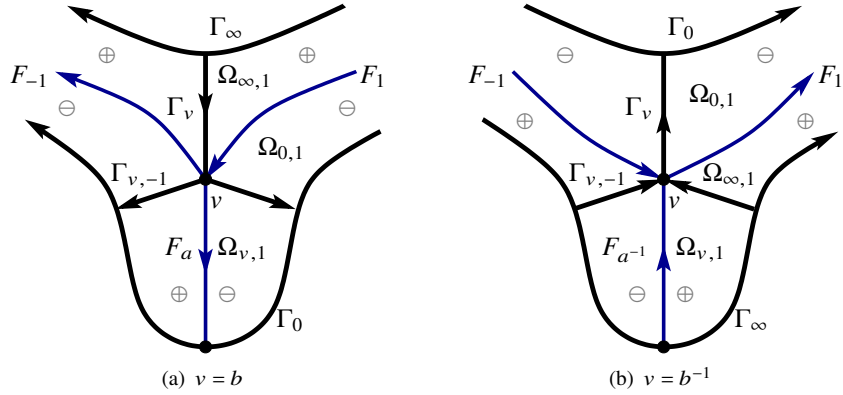
Moreover, we also introduce open oriented arcs \( \Gamma_{v,1}, \Gamma_{v,-1}, \Gamma_v \) connecting \( v \) to \( \Gamma_0\cup\Gamma_\infty \), \( v\in\big\{b,b^{-1}\big\} \), as shown on Figure~\ref{F-lenses}. Besides the interior domain of \( \Gamma_0 \) and the exterior domain of \( \Gamma_\infty \), the union of the introduced arcs, say \( \Gamma \), together with \( F \) delimits eight domains that we label as on Figure~\ref{F-lenses}. Observe that \( \rho \) has holomorphic and non-vanishing extension to each of these eight domains (we can bring arcs \( \Gamma_0,\Gamma_\infty \) closer to \( F \) if necessary). We assume that all the introduced arcs are smooth. Define
\begin{equation}
\label{X2}
\boldsymbol X(z) := \boldsymbol Y(z)\left\{
\begin{array}{rl}
\left(\begin{matrix} 1 & 0 \\ \pm z^n/\rho(z) & 1  \end{matrix}\right), & z\in\Omega_\pm, \medskip \\
\boldsymbol I, & \text{otherwise}.
\end{array}
\right.
\end{equation}
where \( \Omega_+ := \Omega_{\infty,1}\cup\Omega_{\infty,-1}\cup\Omega_{b,-1}\cup\Omega_{b^{-1},1} \) and \( \Omega_- := \Omega_{0,1}\cup\Omega_{0,-1}\cup\Omega_{b,1}\cup\Omega_{b^{-1},-1} \). Then the Riemann-Hilbert problem for \( \boldsymbol X \) can be formulated as follows:
\begin{itemize}
\label{rhx2}
\item[(a)] $\boldsymbol X$ is analytic in $\C\setminus (F\cup\Gamma)$ and $\displaystyle \lim_{z\to\infty} {\boldsymbol X}(z)z^{-n\sigma_3} = \boldsymbol I$;
\item[(b)] $\boldsymbol X$ has continuous traces at the smooth points of $F\cup\Gamma$ that satisfy
\[
\boldsymbol X_+(s) =\boldsymbol X_-(s) \left(\begin{matrix} 0 & \rho(s)/s^n \\ -s^n/\rho(s) & 0 \end{matrix}\right),
\]
for \(s\in F^\circ \), as well as
\[
\boldsymbol X_+(s) =\boldsymbol X_-(s) \left(\begin{matrix} 1 & 0 \\ \pm s^n/\rho(s) & 1  \end{matrix}\right),
\]
for \( s\in\big(\Gamma_0\cup\Gamma_\infty\big)\setminus\big\{a,a^{-1}\big\} \), where we need to use the sign \( - \) for the portion of \(\Gamma_0 \) bordering \( \Omega_{b,-1} \) and the part of \( \Gamma_\infty \) bordering \( \Omega_{b^{-1},-1} \), and
\[
\boldsymbol X_+(s) =  \boldsymbol X_-(s) \left(\begin{matrix} 1 & 0 \medskip \\ s^nR(s) & 1 \end{matrix}\right),
\]
for  \( s\in \cup_{v\in\{b,b^{-1}\}}\big(\Gamma_v \cup \Gamma_{v,1} \cup \Gamma_{v,-1}\big) \), where we put \( \rho_e:=\rho_{|F_e} \) for \( e\in\big\{a,a^{-1},1,-1\big\} \) and set
\[
R(s) := \left\{
\begin{array}{ll}
-\rho_{a^{\pm1}}(s)/(\rho_1\rho_{-1})(s), & s\in\Gamma_{b^{\pm1}}, \smallskip \\
\rho_1(s)/(\rho_{-1}\rho_{a^{\pm1}})(s), & s\in\Gamma_{b^{\pm1},-1}, \smallskip \\
\rho_{-1}(s)/(\rho_{a^{\pm1}}\rho_1)(s), & s\in\Gamma_{b^{\pm1},1};
\end{array}
\right.
\]
\item[(c)] \( \boldsymbol X \) satisfies \hyperref[rhy]{\rhy}(c) except around \( a \) within \( \Omega_{b,1}\cup\Omega_{b,-1} \) where it behaves like \eqref{Xc}.
\end{itemize}

With the above changes, Lemma~\ref{lem:rhx} holds with \eqref{X} replaced by \eqref{X2}.

\subsection{Model and Local RH Problems}

Model Riemann-Hilbert problem \hyperref[rhn]{\rhn} is formulated and solved exactly as in the case \( \rho\in\mathcal W_2 \) for \( n\in \N_\varepsilon \). 

Let now \( U_e \) be an open set around \( e\in E \). In the case of \( U_{b^{\pm1}} \) we shall further assume that these sets do not intersect \( \Gamma_0 \cup \Gamma_\infty \) and completely contain \(  \Gamma_{b^{\pm1}},\Gamma_{b^{\pm1},1}, \Gamma_{b^{\pm1},-1} \), see Figure~\ref{F-local}.
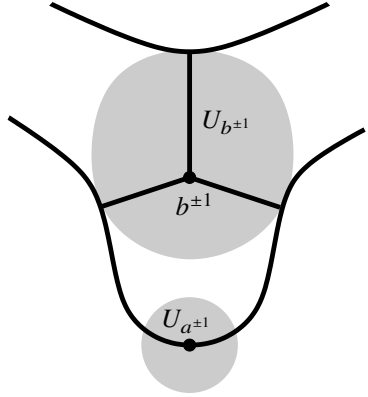
\begin{figure}[ht!]
\centering
\includegraphics[scale=.5]{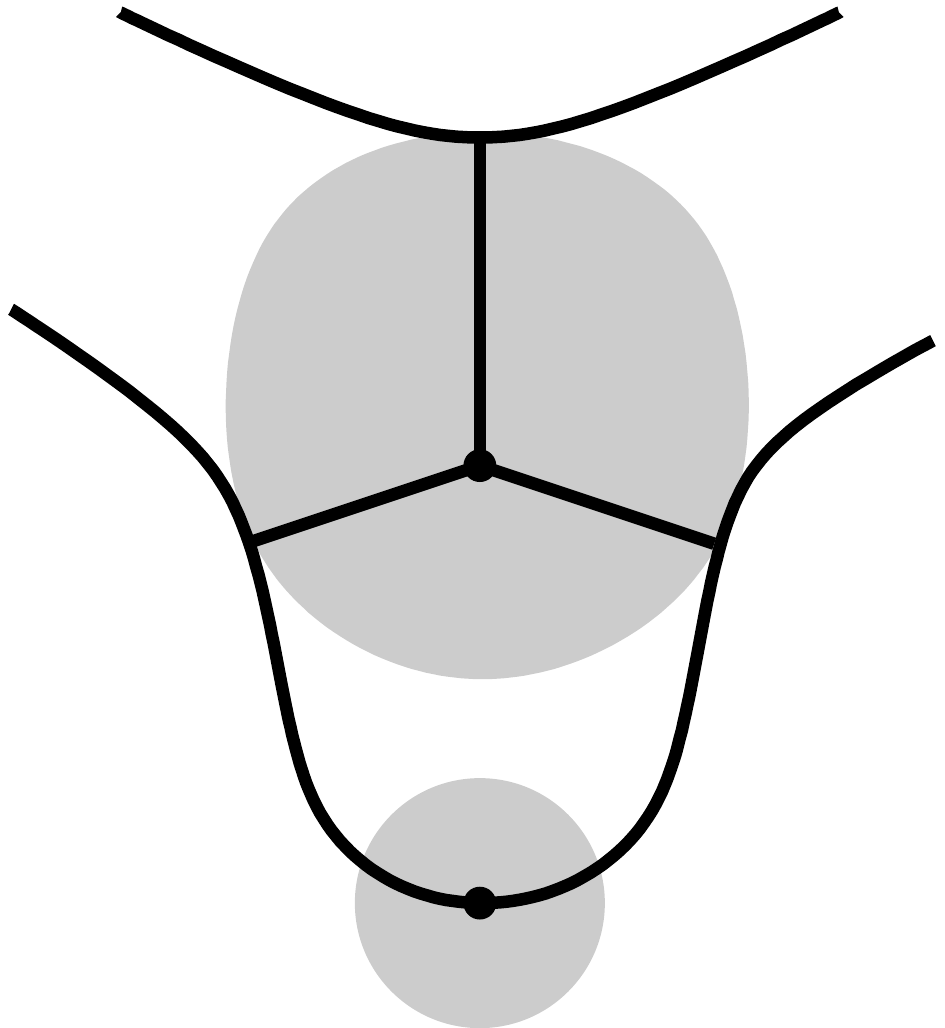}
\begin{picture}(0,0)
\put(-65,100){$U_{b^{\pm1}}$}
\put(-80,25){$U_{a^{\pm1}}$}
\put(-75,67){$b^{\pm1}$}
\end{picture}
\caption{\small Schematic representation of the open sets \( U_e \) for \( e\in E \).}
\label{F-local}
\end{figure}
Define \( \boldsymbol D(z) := \Phi(\z)^{n\sigma_3} \), \( \z\in D_{\mathcal Q} \), where \( \Phi \) is given by \eqref{Phi} and the open set \( D_{\mathcal Q} \) was defined just before Theorem~\ref{thm:NS2}. As in Section~\ref{ssec:lrh} we shall need to solve \hyperref[rhp]{\rhp} for all \( e\in E \). As in the previous proof, we postpone the construction of these matrix-functions until the end of the section.

\subsection{RH Problem with Small Jumps} 

Let
\[
\Sigma = \bigcup_{e\in E} \partial U_e \cup \left[\big( \Gamma_0\cup \Gamma_\infty)\setminus\bigcup_{e\in E} \overline U_e \right].
\]
The Riemann-Hilbert problem \hyperref[rhz]{\rhz} now needs to be formulated as follows:
\begin{itemize}
\label{rhz2}
\item[(a)] $\boldsymbol{Z}$ is a holomorphic matrix function in $\overline\C\setminus\Sigma$ and $\boldsymbol{Z}(\infty)=\boldsymbol{I}$;
\item[(b)] $\boldsymbol{Z}$ has continuous traces on $\big(\Gamma_0\cup \Gamma_\infty\big)\setminus\bigcup_{e\in E} \overline U_e $ that satisfy
\[
\boldsymbol Z_+(s)  = \boldsymbol Z_- (s)\boldsymbol M(s) \left(\begin{matrix} 1 & 0 \\ \pm s^n/\rho(s) & 1 \end{matrix}\right) \boldsymbol M^{-1}(s),
\]
where the choice of the sign \( \pm \) is the same as in the second relation in  \hyperref[rhx2]{\rhx}(b); and
\[
\boldsymbol Z_+(s)  = \boldsymbol Z_- (s)\big(\boldsymbol P_e \boldsymbol M^{-1}\big)(s)
\]
on \( \partial U_e \) for each \( e\in E \).
\end{itemize}

Then the following lemma takes place.
\begin{lemma}
\label{lem:rhz2}
For \( n\in\N_\varepsilon \) large enough, a solution of \hyperref[rhz2]{\rhz} exists and satisfies
\(
\boldsymbol{Z}=\boldsymbol{I}+ \mathcal O\big( 1/n\big)
\)
where $\mathcal O(\cdot)$ holds uniformly in $\overline\C$ and depends on \( \varepsilon \).
\end{lemma}
\begin{proof}
The proof of the fact that the jump of \( \boldsymbol Z \) is geometrically small on \( \big(\Gamma_0\cup \Gamma_\infty\big)\setminus\bigcup_{e\in E} \overline U_e \) is the same as in the case \( \rho\in\mathcal W_2 \). Furthermore, we have that
\[
\boldsymbol P_e \boldsymbol M^{-1} = \boldsymbol I + \boldsymbol{MD}^{-1}\mathcal O(1/n)\boldsymbol{DM}^{-1}
\]
on \( \partial U_e \). It follows from \eqref{calQn}, \eqref{calRn}, \eqref{Psin}, and the equality \( \Phi(\z)\Phi(\z^*)=z \), see Lemma~\ref{lem:Phi}, that the first row of \( \boldsymbol{MD}^{-1} \) is equal to
\[
\left( \begin{matrix} (A_{n\tau+m_n}S_h\Theta_n)(\z) & \pm (A_{n\tau+m_n}S_h\Theta_n)(\z^*) \end{matrix} \right), \quad \z\in D_{\mathcal Q}.
\]
It was shown in the course of the proof of Theorem~\ref{thm:NS2} that these functions have uniformly bounded above moduli on compact subsets of \( \C \). Similarly, one can show that the same is true for the second row of \( \boldsymbol{MD}^{-1} \) as well. Since
\[
\det\big(\boldsymbol{MD}^{-1}\big) = \det(\boldsymbol M) = \big(\gamma_n\gamma_{n-1}^\star\big)^{-1},
\]
and the constants \( |\gamma_n\gamma_{n-1}^\star| \)  are uniformly bounded above for all \( n\in\N_\varepsilon \), we get that the jump of \( \boldsymbol Z \) is of order  \( \boldsymbol I + \mathcal O(1/n) \) on \( \partial U_e \) for each \( e\in  E \). The conclusion of the lemma now follows as in the case \( \rho\in\mathcal W_2 \).
\end{proof}

\subsection{Asymptotics} 

Formulae \eqref{asymptotics2} follow now exactly as in the case \( \rho\in\mathcal W_2 \).

\subsection{Solution of \hyperref[rhp]{\rhp} for \( e\in\big\{a,a^{-1}\big\} \)}

As in Section~\ref{ssec:hard}, we shall only construct the matrix \( \boldsymbol P_a \). The construction is still based on the matrix function \( \boldsymbol\Psi_\alpha \) solving \hyperref[rhpsiA]{\rhpsiA}. Again, we start by defining a special conformal map around \( a \).

\subsubsection{Conformal Map}

With the notation used in Section~\ref{ssec:ND}, define
\begin{equation}
\label{zeta-a-2}
\zeta_a(z) := \left(-\int_a^z\frac{v(s)}{4s}\dd s\right)^2, \quad z\in U_a.
\end{equation}
Since $v_+=-v_-$ on $F_a^\circ$, \( \zeta_a \) is holomorphic in $U_a$. Moreover, since $v$ has a square-root singularity at $a$, $\zeta_a$ has a simple zero at $a$. Thus, we can choose $U_a$ small enough so that $\zeta_a$ is conformal in $\overline U_a$. Recall that \( v(s)\dd s/s \) is purely imaginary on \( F_a^\circ \), see the last part of the proof of Lemma~\ref{lem:Phi}. Therefore, \( \zeta_a \) maps \( F_a \) into the negative reals. As we have had some freedom in choosing the curve \( \Gamma_0 \), we shall choose it within \( U_a \) so that the part of \( \Gamma_0 \) bordering \( \Omega_{b,1} \), say \( \Gamma_0^+ \), is mapped into \( I_+ \) and the part bordering \( \Omega_{b,-1} \), say \( \Gamma_0^- \), is mapped into \( I_- \). Notice that the orientation of \( \zeta_a(\Gamma_0^-) \) is the opposite from the one of \( I_- \).

In what follows, we understand under \( \zeta_a^{1/2} \) the branch given by the expression in parenthesis in \eqref{zeta-a-2}. Equations \eqref{Nuttall} and \eqref{Phi} yield that
\[
\zeta_a(z) = \left(\frac14\log\left(\Phi\big(z^{(0)}\big)/\Phi\big(z^{(1)}\big)\right)\right)^2, \quad z\in U_a.
\]
As \( \Phi(\z)\Phi(\z^*)\equiv z \), relation \eqref{E-zeta-a-1} remains valid for \( \zeta_a \) as above and \( \boldsymbol D(z)=\Phi(\z)^{n\sigma_3} \), \( \z\in D_\mathcal{Q} \). Finally, as in the case of real \( a \), it still holds that \( \zeta_{a+}^{1/4} = \ic \zeta_{a-}^{1/4} \) on \( F_a \).

\subsubsection{Matrix $\boldsymbol P_a$}

Let \( J_a \) be the arc in \( U_a \) emanating from \( a \) such that \( \zeta_a(J_a)\subset[0,\infty) \). According to the conditions placed on the class \( \mathcal W_1 \), it holds that
\[
\rho(z)=\rho_*(z)(z-a)^\alpha,
\]
where $\rho_*$ is non-vanishing and holomorphic in $U_a$ and \( (z-a)^\alpha \) is the branch holomorphic in \( U_a\setminus J_a \). Set \( U_a^\pm \) to be connected components of \( U_a\setminus(F_a\cup J_a) \) containing \( \Gamma_0^\pm \). Define
\[
r_a(z) := \sqrt{\rho_*(z)}(a-z)^{\alpha/2},
\]
where the branch \( (a-z)^{\alpha/2} \) is holomorphic in \( U_a\setminus F_a \) and chosen so
\[
 (a-z)^\alpha = e^{\pm\pi\ic\alpha}(z-a)^\alpha, \quad z\in U_a^\pm.
\]
Then $r_a$ is a holomorphic and non-vanishing function in $U_a\setminus F_a$ and satisfies
\[
\left\{
\begin{array}{ll}
r_{a+}(s)r_{a-}(s) = \rho(s), & s\in F_a\cap U_a, \medskip \\
r_a^2(z) = \rho(z)e^{\pm\pi\ic\alpha}, & z\in\Gamma_0^\pm.
\end{array}
\right.
\]
It can be readily verified now that a solution of \hyperref[rhp]{\rhp} for \( e=a \) is given by \eqref{Pa}, \eqref{Ea}.

\subsection{Solution of \hyperref[rhp]{\rhp} for \( e\in\big\{b,b^{-1}\big\} \)}

We shall construct \( \boldsymbol P_{b^{-1}} \) only as the construction of \( \boldsymbol P_b \) is almost identical.

\subsubsection{Model Problem}

Recall \eqref{rays}. Let  $\boldsymbol\Psi$ be a matrix-valued function such that
\begin{itemize}
\label{rhpsi}
\item[(a)] $\boldsymbol\Psi$ is holomorphic in $\C\setminus\big(I_+\cup I_-\cup(-\infty,\infty)\big)$;
\item[(b)] $\boldsymbol\Psi$ has continuous traces on $I_+\cup I_-\cup(-\infty,0)\cup(0,\infty)$ that satisfy
\[
\boldsymbol\Psi_+ = \boldsymbol\Psi_-
\left\{
\begin{array}{rll}
\left(\begin{matrix} 0 & 1 \\ -1 & 0 \end{matrix}\right) & \text{on} & (-\infty,0), \medskip \\
\left(\begin{matrix} 1 & 0 \\ 1 & 1 \end{matrix}\right) & \text{on} & I_\pm, \medskip \\
\left(\begin{matrix} 1 & 1 \\ 0 & 1 \end{matrix}\right) & \text{on} & (0,\infty);
\end{array}
\right.
\]
\item[(c)] \( \boldsymbol\Psi(\zeta)=\mathcal O(1) \) as $\zeta\to0$;
\item[(d)] $\boldsymbol\Psi$ has the following behavior near $\infty$:
\[
\boldsymbol\Psi(\zeta) = \boldsymbol S(\zeta) \left(\boldsymbol I+\mathcal{O}\left(\zeta^{-3/2}\right)\right) \exp\left\{-\frac23\zeta^{3/2}\sigma_3\right\}
\]
uniformly in $\C\setminus\big(I_+\cup I_-\cup(-\infty,\infty)\big)$, where \( \boldsymbol S(\zeta) \) was defined in \hyperref[rhpsiA]{\rhpsiA}(d).
\end{itemize}

Such a matrix function was constructed in \cite{DKMLVZ99b} with the help of Airy functions.

\subsubsection{Conformal Map}

With the notation used in Section~\ref{ssec:ND}, define
\begin{equation}
\label{zeta-b}
\zeta_{b^{-1}}(z) := \left(-\frac34\int_{b^{-1}}^z\frac{v(s)}{s}\dd s\right)^{2/3}, \quad z\in U_{b^{-1}}.
\end{equation}
Because $v_+=-v_-$ on $F_{a^{-1}}^\circ$, \( \zeta_{b^{-1}}^3 \) is holomorphic in $U_{b^{-1}}$. Moreover, since \( v \) vanishes as a square root when \( z\to b^{-1} \), \( \zeta_{b^{-1}}^3 \) has a cubic zero at \( b^{-1} \) and therefore \( \zeta_{b^{-1}} \) is holomorphic in \( U_{b^{-1}} \). The size of \( U_{b^{-1}}  \) can be adjusted so that \( \zeta_{b^{-1}}  \) is conformal in \( U_{b^{-1}}  \). Recall that the integral of \( v(s)\dd s/s \) is purely imaginary on \( F \). Hence, we can select such a branch in \eqref{zeta-b} that
\[
\zeta_{b^{-1}}(F_{a^{-1}}\cap U_{b^{-1}}) \subset (-\infty,0].
\]
Moreover, we always can adjust the system of arcs \( \Gamma \) so that
\[
\zeta_{b^{-1}}(\Gamma_{b^{-1},-1})\subset I_+, \quad \zeta_{b^{-1}}(\Gamma_{b^{-1},1})\subset I_-, \qandq \zeta_{b^{-1}}(\Gamma_{b^{-1}})\subset (0,\infty).
\]
In what follows, we understand under \( \zeta_{b^{-1}}^{3/2} \) the branch given by the expression in parenthesis in \eqref{zeta-b} and select the branch of \( \zeta_{b^{-1}}^{1/4} \) with the cut along \( F_{a^{-1}} \) satisfying \( \zeta_{b^{-1}+}^{1/4}=\ic \zeta_{b^{-1}-}^{1/4} \). 

Let \( z\in U_{b^{-1}}\setminus (F_{a^{-1}}\cup F_1) \) belong to the component containing \( F_{-1} \), say \( U_{b^{-1}}^1 \), see Figures~\ref{F-lenses} and~\ref{F-local}. Let \( \gamma \) be a path from \( a \) to \( b^{-1} \) and \( \gamma_z \) be a path from \( b^{-1} \) to \( z \) that lie entirely in \( U_{b^{-1}}^1 \). As usual denote by \( B^{(i)} \) the lift of the set \( B \) to \( \RS^{(i)} \). Then
\begin{equation}
\label{paths1}
\gamma^{(0)}\cup\gamma_z^{(0)} \qandq \gamma^{(0)}\cup\boldsymbol\alpha\cup\boldsymbol\beta\cup\gamma_z^{(1)}
\end{equation}
are paths from \( \boldsymbol a \) to \( z^{(0)} \) and \( z^{(1)} \), respectively, that belong to \( \RS_{\boldsymbol\alpha,\boldsymbol\beta} \) (technically, \( \boldsymbol\alpha,\boldsymbol\beta \) in \eqref{paths1} need to be deformed into homologous cycles that belong to \( \RS_{\boldsymbol\alpha,\boldsymbol\beta} \)). Then by using \eqref{paths1} in \eqref{Phi} and recalling \eqref{periods}, we get that
\[
\Phi\big(z^{(0)}\big)/\Phi\big(z^{(1)}\big) = \exp\left\{2\pi\ic(\omega-\tau)+\frac43\zeta_{b^{-1}}^{3/2}(z)\right\}.
\]
Let now \(z \in U_{b^{-1}}\setminus (F_{a^{-1}}\cup F_1) \) be in the component that does not contain \( F_{-1} \), say \( U_{b^{-1}}^2 \). Choose \( \gamma_z \) to be a part of this component. Then
\begin{equation}
\label{paths2}
\gamma^{(0)}\cup\gamma_z^{(0)} \qandq \gamma^{(0)}\cup\boldsymbol\alpha\cup-\boldsymbol\beta\cup\gamma_z^{(1)}
\end{equation}
are paths from \( \boldsymbol a \) to \( z^{(0)} \) and \( z^{(1)} \), respectively, that belong to \( \RS_{\boldsymbol\alpha,\boldsymbol\beta} \) (with the same caveat as before). Thus, we get from \eqref{paths2}, \eqref{Phi}, and \eqref{periods} that
\[
\Phi\big(z^{(0)}\big)/\Phi\big(z^{(1)}\big) = \exp\left\{-2\pi\ic(\omega+\tau)+\frac43\zeta_{b^{-1}}^{3/2}(z)\right\}.
\]
Altogether, we get that
\begin{equation}
\label{E-zeta-b}
\exp\left\{-\frac23n\zeta_{b^{-1}}^{3/2}(z)\right\} = \big(\boldsymbol{KJD}\big)(z)z^{-n\sigma_3/2}\boldsymbol J^{-1}(z),
\end{equation}
where
\[
\boldsymbol J(z) = \left\{
\begin{array}{rl}
\left(\begin{matrix} 0 & 1 \\ -1 & 0 \end{matrix} \right), & z\in D_0\cap U_{b^{-1}}, \medskip \\
\boldsymbol I, & \text{otherwise},
\end{array}
\right.
\]
and
\[
\boldsymbol K(z) :=\left\{\begin{array}{rl}
e^{\pi\ic(\omega-\tau)n\sigma_3}, & z\in U_{b^{-1}}^1, \medskip \\
e^{-\pi\ic(\omega+\tau)n\sigma_3}, & z\in U_{b^{-1}}^2.
\end{array}
\right.
\]

\subsubsection{Matrix \( \boldsymbol P_{b^{-1}} \)}

Set
\[
r_{b^{-1}}(z) := \left\{
\begin{array}{ll}
\sqrt{(\rho_1\rho_{a^{-1}})(z)/\rho_{-1}(z)}, & z\in U_{b^{-1}}^2, \medskip \\
\sqrt{(\rho_1\rho_{-1})(z)/\rho_{a^{-1}}(z)}, & z\in D_0 \cap U_{b^{-1}}^1, \medskip \\
\sqrt{(\rho_{a^{-1}}\rho_{-1})(z)/\rho_1(z)}, & z\in D_\infty \cap U_{b^{-1}}^1.
\end{array}
\right.
\]
Then it follows from \hyperref[rhpsi]{\rhpsi}(a,b,c) that
\[
\boldsymbol P_{b^{-1}}(z) := \boldsymbol E_{b^{-1}}(z)\boldsymbol \Psi\big(n^{2/3}\zeta_{b^{-1}}(z)\big)\boldsymbol J(z)z^{n\sigma_3/2}r_{b^{-1}}^{-\sigma_3}(z)
\]
satisfies \hyperref[rhp]{\rhp}(a,b,c) for \( e=b^{-1} \), where \( \boldsymbol E_{b^{-1}} \) is a holomorphic matrix in \( U_{b^{-1}} \). Thus, it only remains to choose \( \boldsymbol E_{b^{-1}} \) so that \hyperref[rhp]{\rhp}(d) is fulfilled. Set
\[
\boldsymbol E_{b^{-1}}(z) := \boldsymbol M(z) \boldsymbol D^{-1}(z)r_{b^{-1}}^{\sigma_3}(z)\boldsymbol J^{-1}(z)\boldsymbol K^{-1}(z) \boldsymbol S^{-1}\big(n^{2/3}\zeta_{b^{-1}}(z)\big).
\]
Recall that \( \boldsymbol D(z)=\Phi^{n\sigma_3}(\z) \) for \( \z\in D_{\mathcal Q} \). Denote the \( (1,1) \)-entry of \( \boldsymbol D(z) \) by \( d(z) \). Then
\begin{equation}
\label{D-jumps}
(d_-d_+)(s) = \left\{
\begin{array}{rl}
\Phi^n\big(s^{(1)}\big)\Phi^n\big(s^{(0)}\big), & s\in F_{-1}, \medskip \\
\Phi_+^n\big(s^{(1)}\big)\Phi_+^n\big(s^{(0)}\big), & s\in F_1, \medskip \\
\Phi_+^n\big(\s\big)\Phi_+^n\big(\s^*\big), & s\in F_{a^{-1}},
\end{array}
\right. = s^n\left\{
\begin{array}{rl}
1, & s\in F_{-1}, \medskip \\
e^{2\pi\ic\omega n}, & s\in F_1, \medskip \\
e^{2\pi\ic\tau n}, & s\in F_{a^{-1}},
\end{array}
\right.
\end{equation}
by the property \( \Phi(\z)\Phi(\z^*) = z \) and \eqref{Phi-jumps}, where the traces of \( d(z) \) are taken on the arcs in the complex plane and the traces of \( \Phi(\z) \) are taken on the cycles on \( \RS \). Using \hyperref[rhn]{\rhn}(b), \eqref{S}, \eqref{D-jumps}, and the explicit definitions of \( r_{b^{-1}} \), \( \boldsymbol J \), and \( \boldsymbol K \), it is tedious but straightforward to check that \( \boldsymbol E_{b^{-1}} \) is holomorphic in \( U_{b^{-1}}\setminus \big\{b^{-1}\big\} \). It also follows from \eqref{psin-endp} and the behavior of \( \boldsymbol S \) at zero that
\[
\boldsymbol E_{b^{-1}}(z) = \left(\begin{matrix} 1 & |z-b^{-1}|^{-1/2} \medskip \\ 1 & |z-b^{-1}|^{-1/2}  \end{matrix}\right)
\]
which yields that it is, in fact, holomorphic on the whole set \( U_{b^{-1}} \). The relation \hyperref[rhp]{\rhp}(d) now follows from \hyperref[rhpsi]{\rhpsi}(d) and \eqref{E-zeta-b}.

\end{document}